\documentclass{amsart}

\usepackage[english]{babel}
\usepackage[latin1]{inputenc}
\usepackage{enumerate} % \begin{enumerate}[(i)]
\usepackage{latexsym} % symbols
\usepackage{amsthm} % thm etc
\usepackage{amssymb} % symbols
\usepackage{amsmath} % pmatrix
\usepackage{verbatim} % \begin{comment}
\usepackage{fancyhdr}
\usepackage{yfonts}
\usepackage[all]{xy}
\usepackage{mathbbol}

\newtheorem{thm}{Theorem}[section]
\newtheorem{thm2}{Theorem}
\newtheorem{lemma}[thm]{Lemma}
\newtheorem{kor}[thm]{Corollary}
\newtheorem{prop}[thm]{Proposition}
\newtheorem{Def}[thm]{Definition}
\newtheorem{rem}[thm]{Remark}
\newtheorem{eks}[thm]{Example}
\newtheorem*{notation}{Notation}

\newcommand{\R}{\mathbb{R}}
\newcommand{\C}{\mathbb{C}}

\newcommand{\Z}{\mathbb{Z}}
\newcommand{\N}{\mathbb{N}}
\renewcommand{\H}{\mathbb{H}}

\newcommand{\e}{\varepsilon}

\DeclareMathOperator{\id}{id}

\begin{document}

\title[Almost commuting self-adjoint matrices]{Almost commuting self-adjoint matrices --- the real and self-dual cases}
\author{Terry A.Loring}
\address{University of New Mexico\\
Department of Mathematics and Statistics \&
the Center for Advanced Research Computing\\
Albuquerque, New Mexico, 87131, USA}
\email{loring@math.unm.edu}

\author{Adam P. W.S{\o}rensen}
\address{Department of Mathematical Sciences\\
University of Copenhagen\\
Universitetsparken 5, DK-2100, Copenhagen \O, Denmark}
\curraddr{University of New Mexico\\
Department of Mathematics and Statistics\\
Albuquerque, New Mexico, 87131, USA}
\email{apws@math.ku.dk}


\begin{abstract}
We show that a pair of almost commuting self-adjoint, symmetric matrices is close to a pair of commuting self-adjoint, symmetric matrices (in a uniform way). 
Moreover we prove that the same holds with self-dual in place of symmetric. 
The notion of self-dual Hermitian matrices is important in physics when studying fermionic systems that have time reversal symmetry.
Since a symmetric, self-adjoint matrix is real, we get a real version of Huaxin Lin's famous theorem on almost commuting matrices.
Similarly the self-dual case gives a version for matrices over the quaternions. 

We prove analogous results for element of real $C^*$-algebras of ``low rank.''  
In particular, these stronger results apply to paths of almost commuting Hermitian matrices that are real or self-dual.
Along the way we develop a theory of semiprojectivity for real $C^*$-algebras.
\end{abstract}

\maketitle

%%%%%%%%%%%
%%%%%%%%%%%

%%%%%%%%%%%
%%%%%%%%%%%

\section{Introduction}

In 1997 Lin proved an important theorem about almost
commuting matrices \cite{LinAlmostCommutingHermitian}. Nowadays
it is know as Lin's theorem. Loosely speaking it states
that two almost commuting self-adjoint matrices are close to
commuting self-adjoint matrices, and in a way that is
uniform over all dimensions. Formally it says:

\begin{thm} {\rm(Lin)}
For all $\e > 0$ there exists a $\delta > 0$ such that for
all $n \in \N$ the following holds: Whenever $A,B \in M_n(\C)$ are
two self-adjoint, contractive matrices such that $\|AB - BA\| < \delta$ there
exists self-adjoint matrices $A',B' \in M_n(\C)$ such
that $A'B' = B'A'$ and
\[
	\|A - A'\|, \|B - B'\| < \e.
\]
\end{thm}

Friis and R{\o}rdam \cite{friisRordam} generalized Lin's theorem to hold for
elements in a variety of $C^*$-algebras of ``low dimension.'' In
particular, their result applies to paths of almost commuting matrices.

As one might expect, such a fundamental result in linear algebra as Lin's Theorem has had applications outside of $C^*$-algebras.  We believe that many more applications are possible if Lin's theorem is generalized in a few directions.

Applications of almost commuting matrices to quantum mechanics were considered by von Neumann \cite{vonNeumannQuantumErgodic}.  The claim made by von Neumann was not carefully stated, but his underlying assertion regarding macrocopic observables was proven recently by Ogata \cite{ogata2013approximating}, using the techniques in Lins' proof.  The applications of Lin's theorem, and related results on almost commuting unitary matrices, was observed by Hastings \cite{hastings2008topology} and then pursued in a series of papers \cite{HastingsLoringWannier, HastLorTheoryPractice, LoringHastingsEPL}.

More recently, M.M.~Bronstein and his coauthors \cite{kovnatsky2013coupled,glashoff2013almost}
have been applying Lin's theorem, including the real version, to problems in machine learning.  Lin's theorem is used there to prove results regarding commutators measured by the Frobenius norm
\[
 \| A \|_\mathrm{F} = \left( \sum_j \sum_k |a_{jk}|^2  \right)^\frac{1}{2} .
\]

Computer scientist and physicists are less familiar with Lin's theorem than they are with its close relative called ``Joint Approximate Diagonalization'' (JADE).
This algorithm, developed by Cardoso and Souloumiac for use in blind source separation \cite{cardoso1998blind}, takes two or more matrices and finds a unitary change of basis to make both matrices approximately diagonal.  Equivalently \cite{glashoff2013almost}, the algorithm finds commuting matrices as close as possible to the original matrices, as measured in the Frobenius norm.

This is closely related to the problem of finding a small perturbation of an almost commuting pair of matrices to a commuting pair.
JADE minimizes off-diagonal parts in a sense other than as measured by the operator norm.  There is a real version of the algorithm that is probably used far more often than the complex version.  The real version of the Cardoso and Souloumiac algorithm is used in computer vision \cite{glashoff2012multimodal}.  It has also been used in first principal molecular dynamics simulations \cite{gygi2003computation}.   In that case, it is to produce (generalized) Wannier functions.

In the present context, Wannier functions are not really functions at all, but rather joint approximate eigenvectors.  The translation from joint approximate eigenvectors to Wannier functions is discussed at length in \cite{HastLorTheoryPractice}.  Suffice it to say, the smaller the error in the eigenequation in the finite model, the more localized the corresponding Wannier wave function in the valence band.

We see three directions to go in generalizing Lin's theorem so that it might be have broad applications.  The same goes for applications of theorems related to other approximate relations, especially theorems about almost commuting unitary matrices.

The first direction is to deal with anti-unitary symmetries.
The role of anti-unitary symmetries in physics has greatly expanded in the past decade.  Topological insulators were predicted in 2005 and soon thereafter realized experimentally in the form of the quantum spin Hall effect \cite{qi2010quantum}.  Critical here is the time-reversal invariance of the system, where mathematically time-reversal is implemented by an anti-unitary operator, such as the transpose.  Time reversal symmetry differs for Bosons and Fermions, and it can be absent.  The matrices relevant to physics then at least include the real, complex and quaternionic, corresponding to what Dyson called the three-fold way \cite{dyson1962threefold}.

The present paper takes Lins' theorem down two additional paths in the three-fold way.  We note, however, that there can be a second anti-unitary symmetry to be dealt with in some quantum systems, and one needs really to be doing linear algebra with respect to the ten-fold way \cite{altland1997nonstandard} of Atland and Zirnbauer.  A good place to learn about the ten-fold way and its relation to $KO$-theory is \cite{freed2012twisted}.  However, the three-fold way is quite enough for this one paper.

The second direction is to find an algorithm to implement Lin's theorem, in any form.  The operator norm leads to the rather unfamiliar Finsler geometry   \cite{bhatia2007spectral} so this may be very difficult.  However in a one-dimensional setting, where the almost commuting Hermitian matrices $X$ and $Y$ also approximately satisfy $X^2 + Y^2 = I$, there are algorithms available, as discussed in \cite{LoringLogOfUnitary} and \cite[\S 4]{Weigand2012oneDim}.

The third direction is to deal with continuously parametrized sets of almost commuting matrices.  From $K$-theoretical considerations, we must consider only one-dimensional parameter spaces.  An emerging area in topological insulators are the Floquet topological insulators \cite{lindner2011floquet}.  These have a Hamiltonian that varies periodically in time.  Theorem~\ref{realLinPaths} may lead to applications regarding the existence of continuously varying (generalized) Wannier functions in Floquet systems in three symmetry classes.  However, it probably needs to be generalized to deal with paths unitary matrices in place of Hermitian matrices before it finds direct applications.  Some results about almost commuting unitary matrices in the presence of anti-unitary symmetries will be addressed in \cite{LorSorensenTorus,LoringQuantKth}.

Let us focus again on the present paper.  Two paths on the three-fold way involve real and complex matrices.  The third path is seen by the physics world are involving structured complex matrices, but can just as well be seen as involving matrices over the quaternions.  Our results about self-dual
Hermitian matrices thus can be reinterpreted as results about
Hermitian matrices of quaternions.  This is discussed leisurely in \cite{LoringQuaternions}.

We have two main theorems, which we state as one.
\begin{thm2}
\label{mainThm}
For all $\e > 0$ there exists a $\delta > 0$ such that for
all $n \in \N$ the following holds: Whenever $A,B$ are
two $n$-by-$n,$  contractive, self-adjoint, real (resp.~self-dual) matrices such
that $\|AB - BA\| < \delta$ there
exists $n$-by-$n,$  self-adjoint, real (resp.~self-dual) matrices $A',B'$ such
that $A'B' = B'A'$ and
\[
	\|A - A'\|, \|B - B'\| < \e.
\]
\end{thm2}

There are essentially three known proofs of
(the complex case of) Lin's theorem: Lin's original proof and the two proofs given by 
Friis-R{\o}rdam   \cite{friisRordam}.  These later proofs generalize Lin's theorem
beyond matrices.  Working solely with the matrix case, Hastings has derived a
quantitative version of Lin's theorem \cite{hastings2009making}.

The second Friis-R{\o}rdam proof achieves a greater generalization of Lin's theorem than their first proof.  It is the only proof that shows that Lin's theorem remains true for almost commuting elements of  $C^*$-algebras of stable rank one.  This
proof depends on the semiprojectivity of $C(X)$ for $X$ a one-dimensional, finite CW complex \cite{loring96graph}.

We modeled our proof on the second Friis-R{\o}rdam proof.  This requires
that we develop semiprojectivity of real $C^*$-algebras to
to the point of proving the semiprojectivity of $C(X,\mathbb{R})$
for $X$ a one-dimensional, finite CW complex.

We get then a stronger version of our main result, which we present
in Section~\ref{TRrankOne}. This applies to paths of almost commuting matrices.

%%%%%%%%%%%
%%%%%%%%%%%

%%%%%%%%%%%
%%%%%%%%%%%

\section{Real $C^*$-algebras}

%%%%%%%%%%%%%%%%%%%%%%%%%%%%%%%%%%%%%%
%%%%%%%%%%%%%%%%%%%%%%%%%%%%%%%%%%%%%%
\subsection{Two types of real $C^*$-algebras}
%%%%%%%%%%%%%%%%%%%%%%%%%%%%%%%%%%%%%%
%%%%%%%%%%%%%%%%%%%%%%%%%%%%%%%%%%%%%%

In the past, there have been two ways to talk about
real $C^*$-algebras. There have been real $C^*$-algebras
(that is, with lowercase r) and Real $C^*$-algebras
(with uppercase R). For general background on real/Real
$C^*$-algebras, see \cite{goodearlBook} and \cite{liBook}.
Real and real $C^*$-algebras are different objects, and even
though they are closely related the similar names cause
confusion (especially in verbal communication).
We are not the first to feel this way.  See, for example,
\cite[page 698]{may99stable} regarding Atiyah's \cite{atiyah66reality}
use of Real and real as distinct terms. Adding to the confusion
is that fact the Real $C^*$-algebras have $\C$ as their
scalar field. In fairness to Atiyah we should mention that
the category of real spaces sits nicely inside the category of Real spaces,
thus reducing potential confusion. This, however, is not true for
noncommutative $C^*$-algebras. To minimize confusion we suggest new names. Using these names we present some basic facts. A lot of this is either known (though not necessarily written down) or not so hard to prove, so we will omit some proofs.

First we describe a class of algebras with scalar field $\R$. 

\begin{Def}
Given a real Banach $*$-algebra we let $A_\C$ be the set of formal sums $a_1 \dotplus i \cdot a_2$, $a_1, a_2 \in A$. Letting $a_1, a_2, b_1, b_2 \in A$ and $\alpha, \beta \in \R$ we define algebraic operations on $A_\C$ by: 
\begin{align*}
  (a_1 \dotplus i \cdot a_2) + (b_1 \dotplus i \cdot b_2) & = (a_1 + b_1) \dotplus i \cdot (a_2 + b_2), \\
  (a_1 \dotplus i \cdot a_2)(b_1 \dotplus i \cdot b_2)    & = (a_1 b_1 - a_2 b_2) \dotplus i \cdot (a_2 b_1 + a_1 b_2), \\
  (a_1 \dotplus i \cdot a_2)^* & = a_1^* \dotplus i \cdot (-a_2^*), \\
  (\alpha + \beta i)(a_1 \dotplus i \cdot a_2) & = (\alpha a_1 - \beta a_2) \dotplus i \cdot (\alpha a_2 + \beta a_1). 
\end{align*} 
With those operations $A_\C$ is a complex $*$-algebra. It is called the {\em complexification} of $A$. 
\end{Def}

\begin{Def}
A real Banach $*$-algebra $A$ is called an {\em $R^*$-algebra} if there exist a norm on $A_\C$ such that  $A_\C$ becomes a $C^*$-algebra, and the norm on $A_\C$ extends the norm on $A$.
\end{Def}

\begin{rem}
An $R^*$-algebra is known in the literature as a real $C^*$-algebra \cite{PalmerRealCstar}. 
\end{rem}

We have the obvious morphisms, and with those we have a category.  

\begin{Def}
A map $\phi \colon A \to B$ between two $R^*$-algebras is called an {\em $R^*$-homomorphism} if it is $\R$-linear, multiplicative and $*$-preserving.
\end{Def}

\begin{Def}
Denote by $\mathbf{R^*}$ the category with objects all $R^*$-algebras and morphisms all $R^*$-homomorphisms. Denote by $\mathbf{R^*_1}$ the category of unital $R^*$-algebras and unital morphisms. 
\end{Def}

We will also define a class of algebras that is seemingly closer to $C^*$-algebras.
The motivation for this is that the real matrices can
be described as those complex matrices where the adjoint and the transpose agree. 
We define something similar to the transpose in a more general setting. 

\begin{Def}
Let $A$ be a $C^*$-algebra. A linear and $*$-preserving map $\tau \colon A \to A$ such that $\tau(ab) = \tau(b) \tau(a)$, and $\tau(\tau(a)) = a$ for all $a, b \in A$ is called a {\em reflection} on $A$. 
\end{Def}

\begin{rem}
A reflection is an order two anti-automorphism of $A$. Thus it is automatically norm preserving and continuous. Furthermore, the $0$ element in $A$ must be mapped to $0$ by $\tau$, if $A$ has a unit it too must be mapped to itself by $\tau$, and for any $a \in A$ the spectrum of $a$ equals that of $\tau(a)$.     
\end{rem}

\begin{Def}
A  {\em $C^{*,\tau}$-algebra} is a pair $(A, \tau)$ where $A$ is a $C^*$-algebra and $\tau$ is a reflection of $A$. We will often write $\tau(a)$ as $a^\tau$. 
\end{Def}

Similar to how the letter $d$ is almost always used to represent a generic metric, we will write $(A,\tau)$ when we do not know anything special about $\tau$.

\begin{rem}
The Real $C^*$-algebras correspond to $C^{*,\tau}$-algebras. 
\end{rem}

We also have morphisms between $C^{*,\tau}$-algebras, and so we also get a category. 

\begin{Def}
By a {\em $C^{*,\tau}$-homomorphism} (or $*$-$\tau$-homomorphism) we mean a map $\phi \colon (A,\tau) \to (B,\tau)$ such that $\phi$ is a $*$-homomorphism from $A$ to $B$ and $\phi(a^\tau) = \phi(a)^\tau$ for all $a \in A$. 
\end{Def}

\begin{Def}
Let $\mathbf{C^{*,\tau}}$ be the category with objects all $C^{*,\tau}$-algebras and morphisms all $*$-$\tau$-homomorphisms. Let $\mathbf{C^{*,\tau}_1}$ be the category of unital $C^{*,\tau}$-algebras and unital morphisms. 
\end{Def}

%%%%%%%%%%%%%%%%%%%%%%%%%%%%%%%%%%%%
\subsection{Connections between $\mathbf{R^*}$ and $\mathbf{C^{*,\tau}}$}
%%%%%%%%%%%%%%%%%%%%%%%%%%%%%%%%%%%%

We will now consider the close relationship between $R^*$-algebras and $C^{*,\tau}$-algebras. We have a notion of real elements inside a $C^{*,\tau}$-algebra.  

\begin{Def}
Given $a \in (A,\tau)$ we let $\Re_\tau(a) = (a + a^{*\tau})/2$. 
\end{Def}

We will say that $a$ is a real element or that it is in the real part of $(A,\tau)$ if $\Re_\tau(a) = a$. This happens precisely when $a^* = a^\tau$.

\begin{lemma}
If $a \in (A,\tau)$ then 
\[
  a = \Re_\tau(a) - i \Re_\tau(i a).
\]
\end{lemma}

\begin{lemma}
If $a \in (A,\tau)$ and we can write $a = a_1 + ia_2$ with $a_1$ and $a_2$ in the real part of $A$ then $a_1 = \Re_\tau(a)$ and $a_2 = \Re_\tau(-ia)$. 
\end{lemma}

We use this new-found knowledge to show that inside all $C^{*,\tau}$-algebras lives an $R^*$-algebra. 

\begin{prop} \label{realPartIsRStar}
If $(A,\tau)$ is a $C^{*,\tau}$-algebra then $\{ a \in A \mid a^* = a^\tau \}$ is an $R^*$-algebra. 
\end{prop}
\begin{proof}
Let $A_0 = \{ a \in A \mid a^* = a^\tau \}$. The map from $A$ to $(A_0)_\C$ sending $a \in A$ to $\Re_\tau(a) \dotplus i \cdot \Re_\tau(-ia)$ is an $R^*$-isomorphism.
\end{proof}

We now define a functor from $\mathbf{C^{*,\tau}}$ to $\mathbf{R^*}$.

\begin{Def}
Define $\Re \colon \mathbf{C^{*,\tau}} \to \mathbf{R^*}$ on objects by
\[
  \Re((A,\tau)) = \{ a \in A \mid a^* = a^\tau \},
\]
and if $\phi \colon (A,\tau) \to (B,\tau)$ we let
\[
  \Re(\phi) = \phi \vert_{\Re(A,\tau)},
\]
where we co-restrict the right hand side to $\Re((B,\tau))$.
\end{Def}

We also wish to have a functor from $\mathbf{R^*}$ to $\mathbf{C^{*,\tau}}$. 

\begin{lemma} \label{canonicalReflection}
If $A$ is an $R^*$-algebra then $\bar{*} \colon A_\C \to A_\C$ given by
\[
  (a_1 \dotplus i \cdot a_2)^{\bar{*}} = a_1^* \dotplus i \cdot a_2^*,
\]
is a reflection on $A_\C$. Furthermore $\Re(A_\C, \bar{*}) \cong A$.
\end{lemma}

\begin{Def}
Define $\overline{\bigstar}$ to be the functor from $\mathbf{R^*}$ to $\mathbf{C^{*,\tau}}$ that maps an $R^*$-algebra $A$ to $(A_\C, \bar{*})$ and an $R^*$-homomorphism $\phi \colon A \to B$ to $\overline{\bigstar}(\phi) \colon (A_\C, \bar{*}) \to (B_\C, \bar{*})$ given by
\[
  \overline{\bigstar}(\phi)(a_1 \dotplus i \cdot a_2) = \phi(a_1) \dotplus i \cdot \phi(a_2). 
\]
\end{Def}

It is not obvious that $\overline{\bigstar}$ is a functor, but on the other hand it is not hard to prove. 

\begin{rem}
The functor $\overline{\bigstar}$ maps surjections to surjections and injections to injections. 
\end{rem}

It can be shown that our two functors are almost inverses. That is, if $A$ is an $R^*$-algebra and $(B,\tau)$ is a $C^{*,\tau}$-algebra, then 
\[
  \overline{\bigstar}(\Re(B,\tau)) \cong (B,\tau), \quad \text{ and } \quad \Re(\overline{\bigstar}(A)) \cong A.
\]
In fact it is know that they both yield categorical equivalences. As such a lot of the study of $R^*$-algebras can be done using $C^{*,\tau}$-algebras. This is the approach we will take throughout this paper. The reasoning behind this choice is that the $C^{*,\tau}$-algebras lets us utilize a lot of our $C^*$-algebra knowledge. Hence there is less reproving of theorems. 

\subsection{Two examples}

\begin{eks}
We modeled a reflection on the transpose so of course it is a reflection, and $\Re(M_n(\C), T) = M_n(\R)$.

There is another reflection on $M_{2n}(\C)$. If $A \in M_{2n}(\C)$ we let $A_{ij}$ be the $n \times n$ blocks and define 
\[
	\begin{pmatrix}
		A_{11} & A_{12} \\
		A_{21} & A_{22}
	\end{pmatrix}^\sharp 
	=
	\begin{pmatrix}
		A_{22}^T & - A_{12}^T \\
		- A_{21}^T & A_{11}^T
	\end{pmatrix}.
\]
This is a reflection, and $\Re(M_{2n}(\C), \sharp) = M_n(\H)$, where $\H$ is the quaternions. This is an important operation in physics, as
is discussed in the survey \cite{zabrodin2006matrix} of applications of random
matrices in physics.  
\end{eks}

\begin{eks}
Consider the $C^*$-algebra of continuous complex-valued functions on the circle $C(S^1)$. Since $C(S^1)$ is abelian, a reflection is just an order-two isomorphism. Hence any reflection will come from an order-two homeomorphism of the circle. From \cite{circleHom} we glean that there are only three such maps (up to conjugation), namely:
\begin{enumerate}
	\item $z \mapsto -z$, 
	\item $z \mapsto \overline{z}$, and,
	\item $z \mapsto z$.
\end{enumerate}
Each gives rise to a $C^{*,\tau}$-algebra by defining for instance $f^\tau(z) = f(-z)$. The real parts will be 
\begin{enumerate}
	\item $\{ f \in C(S^1, \C) \mid \overline{f(z)} = f(-z) \text{ for all } z \in S^1 \}$, 
	\item $\{ f \in C(S^1, \C) \mid \overline{f(z)} = f(\overline{z}) \text{ for all } z \in S^1 \}$, and, 
	\item $\{ f \in C(S^1, \C) \mid \overline{f(z)} = f(z) \} \cong C(S^1, \R)$.
\end{enumerate}
\end{eks}

As there are two essentially distinct reflections on
$\mathbf{M}_{2n} (\mathbb{C})$
and three on $C(S^1),$ we immediately find
six replacements in the real case for 
\[
U_{2n}(A) \cong \hom \left (C(S^1), \mathbf{M}_{2n}(A) \right ).
\]
Hence we get six objects that could correspond to the 6 odd real $K$-groups, counting degrees $1$, $3$, $5$, and
$7$ in $KO$ and degrees $1$ and $3$ in self-conjugate $K$-theory
\cite{atiyah66reality, BoersemaUnitedKtheory}.

%%%%%%%%%%%%%%%%%%%%%%%%%%%%%%%%%%%%%%
%%%%%%%%%%%%%%%%%%%%%%%%%%%%%%%%%%%%%%
\subsection{Ideals in and operations on}
%%%%%%%%%%%%%%%%%%%%%%%%%%%%%%%%%%%%%%
%%%%%%%%%%%%%%%%%%%%%%%%%%%%%%%%%%%%%%

We wish to study ideals in $C^{*,\tau}$-algebras.
In $C^*$-algebras the ideals are precisely the kernels of $*$-homomorphisms.
The kernel of a $C^{*,\tau}$-homomor-phism %spacing hack
will be self-$\tau$ (that is, if $x \in \ker \phi$ then $x^\tau \in \ker \phi$), but there are $C^*$-ideals that need not be self-$\tau$.
We wish to eliminate those ideals, and so we give the following definition. 

\begin{Def}
Let $(A,\tau)$ be a $C^{*,\tau}$-algebra. We say that $I \subseteq A$ is an {\em ideal} in $(A,\tau)$ if $I$ is a $C^*$-ideal in $A$ and $I$ is self-$\tau$. We will sometimes write $I \triangleleft_\tau A$ or $I \triangleleft (A, \tau)$.
\end{Def}

With ideals at hand, we can define quotients. 

\begin{lemma}
If $I \triangleleft (A,\tau)$ then $(I, \tau \vert_I)$ is a $C^{*,\tau}$-algebra. Let $\pi \colon A \to A/I$ be the quotient map.  The map $\pi(a)^\tau \mapsto \pi(a^\tau)$ defines a reflection on $A/I$. Thus, $A/I$ is naturally a $C^{*,\tau}$-algebra and $\pi$ is $C^{*,\tau}$-homomorphism. 
\end{lemma}

We note that we now have obtained what we wanted: The $C^{*,\tau}$ ideals are precisely the kernels of the $C^{*,\tau}$-homomorphisms.

The following lemma and theorem tells us that we have direct sums and pullbacks in the category $\mathbf{C^{*,\tau}}$.

\begin{lemma}
Given two $C^{*,\tau}$-algebras $(A,\tau)$ and $(B,\tau)$ the map $\tau \oplus \tau \colon A \oplus B \to A \oplus B$ will be a reflection.   
\end{lemma}

\begin{thm}
Suppose $\varphi_1 \colon (A_1,\tau) \to (C,\tau)$ and $\varphi_2 \colon (A_{2},\tau) \to (C,\tau)$
are $*$-$\tau$-homomorphisms, and form the pull-back $C^{*}$-algebra
\[
  A_1\oplus_C A_2 =\left\{ (a_1,a_2) \in A_1\oplus A_2 \mid \varphi_1(a_1) =\varphi_2(a_2) \right\} .
\]
This becomes a $C^{*,\tau}$-algebra with
\[
  (a_{1},a_{2})^{\tau}=(a_{1}^{\tau},a_{2}^{\tau})
\]
and it gives us the pull-back of the given $C^{*,\tau}$-algebras, where we are using the restricted projection maps $\pi_j \colon (A_1,\tau) \oplus_{(C,\tau)} (A_2,\tau) \rightarrow (A_j,\tau).$ 
\end{thm}

We can also define what it means to unitize a $C^{*,\tau}$-algebra. 

\begin{lemma} \label{onlyOneWayToUnitize}
Let $(A,\tau)$ be a $C^{*,\tau}$ algebra. The formula 
\[
 (a + \lambda \mathbb{1})^\sigma = a^\tau + \lambda \mathbb{1}, \quad a \in A, \lambda \in \C,
\]
defines a reflection on $\tilde{A}$. Thus
$(\tilde{A},\sigma)$ is a $C^{*,\tau}$-algebra. And it
is the only way to unitize $(A,\tau)$ while preserving the reflection on $A$.
\end{lemma}

\begin{Def}
If $(A,\tau)$ is a $C^{*,\tau}$-algebra we will also denote
by $\tau$ the extension of $\tau$ to $\tilde{A}$ given in
lemma~\ref{onlyOneWayToUnitize} (this should cause no confusion,
as the lemma shows this extension is unique). The
$C^{*,\tau}$-algebra $(\tilde{A},\tau)$ we denoted
by $\widetilde{(A,\tau)}$ or $(A,\tau)^\sim$, and call it
the {\em unitization} of $(A,\tau)$.
\end{Def}

\begin{eks}
The unitization of the $C^{*,\tau}$-algebra $C_0((0,1),\id)$ is $C(S^1, \id)$.
\end{eks}

%%%%%%%%%%%
%%%%%%%%%%%

%%%%%%%%%%%
%%%%%%%%%%%

\section{(Semi) Projective real $C^*$-algebras}

The definition of a semiprojective $C^*$-algebra commonly used today was given by Blackadar in \cite{blackadar1985shape}. We will modify that definition so we can use it for $C^{*,\tau}$-algebras. The theory of semiprojective $C^*$-algebras is well developed; for good resources on the subject see \cite{blackadar2004semiprojectivity}, \cite{loringBook}, and the references therein. In what follows we try to develop some theory of semiprojective $C^{*,\tau}$-algebras.

Semiprojectivity is the first key ingredient in shape theory \cite{EffrosKaminkerShape}.
The other key ingredient is $K$-theory, which we do not see here. The $K$-theory is uninteresting when working with one-dimensional spaces, as in this section.
In section 6 we are working with the disk or square so the $K$-theory
is still not a factor, this time due to the contractability of these
space.  See \cite{HastLorTheoryPractice} for how $K$-theory will make
life more interesting when we get to the sphere
and torus.

A semiprojective $C^*$-algebra is an analog of an absolute
neighborhood retract (ANR).
Indeed, a necessary, but far from sufficient condition, for $A$ to be
semiprojective is that its abelianization have spectrum an ANR.
Semiprojective $C^*$-algebras interact well with inductive limits.
One approach to shape theory (take by Effros and Kaminker in \cite{EffrosKaminkerShape}) is to write two $C^*$-algebras
of interest as inductive limits of semiprojective $C^*$-algebras,
then intertwine the semiprojective $C^*$-algebras to produce
a substitute for a homotopy equivalence when a true homotopy
equivalence is not possible.

It was later discoverer, by Dadarlat \cite{DadarlatShapeAndE}, that
shape theory is closely related to $E$-theory.  In Dadarlat's approach
to shape theory, semiprojectivity plays only a supporting role.

Neither shape theory nor $E$-theory have been worked out for real $C^*$-algebras, but questions of semiprojectivity and the lack-thereof in real $C^*$-algebras are now seen to arise in physics \cite{HastLorTheoryPractice}.
Without prejudice, we walk around those subjects and attack semiprojectivity for real $C^*$-algebras head-on.

Our task is to incorporate reflections in the subject of
semiprojectivity. Our goal is to prove that $C(X,\id)$ is semiprojective for any finite one dimensional CW complex. We will only achieve that goal once we reach the end of section \ref{FunGraphSec}. 

In this and the next two sections we will
borrow proof techniques from across the field of semiprojectivity
in $C^*$-algebras without further references.

\subsection{Definitions and the basics}

We give the obvious definitions of projectivity and semiprojectivity in the categories $\mathbf{C^{*,\tau}}$ and $\mathbf{C^{*,\tau}_1}$. 

\begin{Def}
Let $\mathbf{C}$ be one of the categories $\mathbf{C^{*,\tau}}$ or $\mathbf{C^{*,\tau}_1}$. An object $A$ in $\mathbf{C}$ is said to be {\em projective}, if whenever $J$ is an ideal in $B$, another object in $\mathbf{C}$, and we have a morphism $\phi \colon A \to B / J$ in $\mathbf{C}$, we can find a morphism $\psi \colon A \to B$ in $\mathbf{C}$ such that $\pi \circ \psi = \phi$, where $\pi$ it the quotient map from $B$ to $B / J$. 
\end{Def}

\begin{Def}
Let $\mathbf{C}$ be one of the categories $\mathbf{C^{*,\tau}}$ or $\mathbf{C^{*,\tau}_1}$. An object $A$ in $\mathbf{C}$ is said to be {\em semiprojective}, if whenever $J_1 \subseteq J_2 \subseteq \cdots$ is an increasing sequence of ideals in $B$, another object in $\mathbf{C}$, and we have a morphism $\phi \colon A \to B / J$, $J = \overline{\cup_n J_n}$, in $\mathbf{C}$, we can find an $m \in \N$ and morphism $\psi \colon A \to B / J_m$ in $\mathbf{C}$ such that $\pi_{m,\infty} \circ \psi = \phi$, where $\pi_{m,\infty}$ it the quotient map from $B / J_m$ to $B / J$. 
\end{Def}

\begin{notation}
Whenever we have a $C^{*,\tau}$-algebra $B$ containing an increasing sequence of $\tau$-invariant ideals $J_1 \subseteq J_2 \subseteq \cdots$ we denote the quotient maps as follows:
\begin{align*}
	& \pi_n \colon B \twoheadrightarrow B/J_n, \\
	& \pi_{n,m} \colon B/J_n \twoheadrightarrow B/J_m, \\
	& \pi_{m,\infty} \colon B/J_m \twoheadrightarrow B/J, \\
	& \pi_\infty \colon B \twoheadrightarrow B / J,
\end{align*}
where $n < m$ are natural numbers and $J = \overline{\cup_n J_n}$. 
\end{notation}

Pictorially, $(A,\tau)$ is semiprojective, if we can always solve the following lifting problem
\[
\xymatrix{
																				&(B,\tau) \ar@{->>}[d] \\
																				&(B/J_n,\tau) \ar@{->>}[d] \\
	(A,\tau) \ar[r]_{\phi} \ar@{-->}[ur]	& (B/J,\tau)
}
\]

Of course one could just as easily define semiprojective $R^*$-algebras. Studying how the functors $\Re$ and $\overline{\bigstar}$  behave with respect to ideals and lifting problems, the following two propositions can be proved. For reasons of brevity we have chosen not to include proofs of these propositions.

\begin{prop}
Let $A, B$ be $R^*$-algebras, let $J$ be an ideal in $B$, and let $\phi \colon A \to B/J$ be an $R^*$-homomorphism.
There exists an $R^*$-homomorphism $\psi \colon A \to B$ such that $\pi \circ \psi = \phi$ if and only if there exists a $C^{*,\tau}$-homomorphism $\chi \colon \overline{\bigstar}(A) \to \overline{\bigstar}(B)$ such that $\overline{\bigstar}(\pi) \circ \chi = \overline{\bigstar}(\phi)$.
\end{prop}

\begin{prop}
If $A$ is an $R^*$-algebra then $A$ is (semi-) projective if and only if $\overline{\bigstar}(A)$ is. If $(B,\tau)$ is a $C^{*,\tau}$-algebra then $(B,\tau)$ is (semi-) projective if and only if $\Re(B,\tau)$ is. 
\end{prop}

Just as in the $C^*$-case we can somewhat simplify the task of proving semiprojectivity. 

\begin{prop} \label{easierSP}
To show that a $C^{*,\tau}$-algebra $(A,\tau)$ is semiprojective it
suffices to solve lifting problems where $\phi \colon (A,\tau) \to (B/J,\tau)$ is either injective, surjective or both.
\end{prop}
\begin{proof}
This is well know in the $C^*$-case \cite{loringBook}, and
is no harder in the $C^{*,\tau}$-case. To get injectivity
we replace $\phi$ with
$\phi \oplus \id \colon (A,\tau) \to (B/J \oplus A, \tau \oplus \tau)$.
To get surjectivity we focus on the image of $\phi.$
\end{proof}

Functional calculus is indispensable when working with lifting problems. The following lemma tells some of the story about $C^{*,\tau}$-algebras and functional calculus. 

\begin{lemma} \label{tauAndFunctions}
Suppose $a$ is a normal element in a $C^{*,\tau}$-algebra $(A,\tau)$. If $f$ is a continuous function from $\sigma(a)$ to $\C$ then $f(a)^\tau = f(a^\tau)$.

If $b \in (A,\tau)$ is a normal and self-$\tau$ element and $\sigma(b) \subseteq X \subseteq \C$ then the $C^*$-homomorphism $\phi \colon C_0(X) \to A$ given by $f \mapsto f(b)$ is a $C^{*,\tau}$-homomorphism from $C(X,\id)$ to $(A,\tau)$. 
\end{lemma}
\begin{proof}
We remind the reader that $\sigma(a) = \sigma(a^\tau)$. Since $\tau$ is linear and $a$ is normal we have $p(a)^\tau = p(a^\tau)$, for any polynomial $p$ in $a$ and $a^*$. By continuity of $\tau$ we now get $f(a)^\tau = f(a^\tau)$ for any function $f \in C(\sigma(a))$.

For any function $f \in C_0(X)$ we have 
\[
	f(b)^\tau = f(b^\tau) = f(b) = (f \circ \id)(b). \qedhere
\]
\end{proof}

With that lemma at our disposal, we can give some basic examples of (semi-) projective $C^{*,\tau}$-algebras. 

\begin{eks} \label{projectiveEx}
The $C^{*,\tau}$-algebra $C_0((0,1], \id)$ is projective. To see this, suppose we are given the following lifting problem:
\[
\xymatrix{
																& (B,\tau) \ar@{->>}[d]^{\pi} \\
	C_0((0,1], \id) \ar[r]_{\phi}	& (B/J,\tau)
}
\]
Let $h = \phi(t \mapsto t)$. Then $h$ is a self-$\tau$ positive contraction. Let $x$ be a positive contractive lift of $h$, and let $k = (x + x^\tau)/2$. Then $k$ is a self-$\tau$, positive contraction, and $\pi(x) = h$. By Lemma~\ref{tauAndFunctions} the map $f \mapsto f(k)$ is a $C^{*,\tau}$-homomorphism. It is a lift of $\phi$ by standard $C^*$-theory.
\end{eks}

\begin{eks}
The $C^{*,\tau}$-algebra $(\C, \id)$ is semiprojective. To prove this, suppose we are given the following lifting problem:
\[
\xymatrix{
													&(B,\tau) \ar@{->>}[d]^{\pi_{n}} \\
													&(B/J_n,\tau) \ar@{->>}[d]^{\pi_{n,\infty}} \\
	( \C,\id) \ar[r]_{\phi}	& (B/J,\tau)
}
\] 
Let $p = \phi(1)$. Then $p$ is a self-$\tau$ projection. Let $y \in (B,\tau)$ be any self-adjoint lift of $p$. If we let $x = (y + y^\tau)/2$ then $x$ is a self-$\tau$ and self-adjoint lift of $p$. Since $\pi_{n}(x^2 - x) \to 0$ as $n \to \infty$ we can find some $m \in \N$ such that $1/2 \notin \sigma(\pi_{m}(x))$. Now let $f$ be the function that is $0$ on $(-\infty;1/2)$ and $1$ on $(1/2;\infty)$. Then $q = f(\pi_m(x))$ is a projection and a lift of $p$. Since $x$ is self-$\tau$, $q$ will be self-$\tau$. We can now define a $C^{*,\tau}$-homomorphism from $\C$ to $B/J_m$ by $\lambda \mapsto \lambda q$ (Lemma~\ref{tauAndFunctions}). It is a lift of $\phi$.   
\end{eks}

\subsection{Closure results}

\subsubsection{Unitizing}

We aim to get the $C^{*,\tau}$ equivalent of $C^*$ result that $A$ is semiprojective if and only if $\tilde{A}$ is. First we show that if $A$ is unital it suffices to solve unital lifting problems. 

\begin{lemma} \label{spIffUnitalSp}
A unital $C^{*,\tau}$-algebra is semiprojective in $\mathbf{C^{*,\tau}}$ if and only if it is semiprojective in $\mathbf{C^{*,\tau}_1}$. 
\end{lemma}
\begin{proof}
Let $(A,\tau)$ be a unital $C^{*,\tau}$-algebra. 

The proof that $(A,\tau)$ semiprojective in $\mathbf{C^{*,\tau}}$ implies that it is semiprojective in $\mathbf{C^{*,\tau}_1}$ is precisely the same as in the $C^*$-case.

Suppose that $(A,\tau)$ is semiprojective in $\mathbf{C^{*,\tau}_1}$. Let $(B,\tau)$ be a $C^{*,\tau}$-algebra containing an increasing sequence of $C^{*,\tau}$ ideals $J_1 \subseteq J_2 \subseteq \cdots$, let $J = \overline{\cup_n J_n}$, and let $\phi \colon (A,\tau) \to (B/J, \tau)$ be a $C^{*,\tau}$-homomorphism. Put $p = \phi(1_A)$. Then $p$ is a self-$\tau$ projection in $(B/J,\tau)$. Since $(\C, \id)$ is semiprojective, we can find some $n_0 \in \N$ and self-$\tau$ projection $q \in B/J_{n_0}$ such that $\pi_{n_0,\infty}(q) = p$. For each $n \geq n_0$ define $q_n = \pi_{n_0,n}(q)$. Since all the $q_n$ are self-$\tau$ all the corners $q_n(B/J_n)q_n$ are self-$\tau$. Hence for all $n \geq n_0$ we have that $q_n(B / J_n)q_n \cong (qBq)/(q J_n q)$ and that by restricting the $\tau$'s we get the following commutative diagram of $C^{*,\tau}$-algebras:
\[
\xymatrix{
	& (q_{n_0} (B / J_{n_0}) q_{n_0},\tau) \ar@{->>}[d] \lhook\mkern-7mu \ar[r]	& (B/J_{n_0}, \tau) \ar@{->>}[d] \\
	& (q_{n} (B / J_{n}) q_{n},\tau) \ar@{->>}[d] \lhook\mkern-7mu \ar[r]	& (B/J_{n}, \tau) \ar@{->>}[d] \\
	(A,\tau) \ar[r]_-{\phi}	& (p (B / J) p,\tau) \lhook\mkern-7mu \ar[r]	& (B/J, \tau) 
}
\] 
In the two left most columns there are only unital maps and algebras, so since $(A,\tau)$ is semiprojective in the unital category, we can find a lift for some $n \geq n_0$. This lifting combines with the inclusion $q_{n} (B / J_{n}) q_{n} \hookrightarrow B /J_n$ to show that $(A,\tau)$ is semiprojective. 
\end{proof}

The lemma is a stepping stone towards a goal, but it also has its own applications. 

\begin{eks}
The $C^{*,\tau}$-algebras $C(S^1, \id), C(S^1, z \mapsto \overline{z})$, and $C(S^1,$ $z \mapsto -z)$ % spacing hack
are all semiprojective. We will only show the first one, but the remaining proofs are similar. By Lemma \ref{spIffUnitalSp} it suffices to solve lifting problems of the form:
\[
\xymatrix{
															&(B,\tau) \ar@{->>}[d]^{\pi_{n}} \\
															&(B/J_n,\tau) \ar@{->>}[d]^{\pi_{n,\infty}} \\
	(C(S^1),\id) \ar[r]_{\phi}	& (B/J,\tau)
}
\] 
where everything is unital. Let $u = \phi(z \mapsto z)$. Then $u$ is a self-$\tau$ unitary. Let $y$ be any self-$\tau$ lift of $u$. We can find an $m$ such that $x = \pi_m(x)$ satisfies that $xx^*$ and $x^*x$ are invertible. Now define $v = x(x^*x)^{-1/2}$. The $v$ is a unitary lift of $u$ and, by Lemma~\ref{tauAndFunctions} and a standard functional calculus trick, 
\[
	v^\tau = ((x^*x)^{-1/2})^\tau x^\tau = ((x^*x)^\tau)^{-1/2} x = (xx^*)^{-1/2} x = x (x^*x)^{-1/2} = v. 
\] 
There is $C^*$-homomorphism from $C(S^1)$ to $B/J_m$ given by $\psi(f) = f(v)$. Since $v$ is self-$\tau$ and every element in $C(S^1, \id)$ is self-$\tau$, this is actually a $C^{*,\tau}$-homomorphism from $C(S^1, \id)$ to $(B/J_m, \tau)$. Because $v$ is a lift of $u$, $\psi$ is a lift of $\phi$.
\end{eks}

We omit the proof of the next lemma, as the standard $C^*$ proof carries over easily. 

\begin{lemma} \label{spIffUnitizationSp}
A $C^{*,\tau}$-algebra $(A,\tau)$ is semiprojective if and only if $\widetilde{(A,\tau)}$ is semiprojective in the unital $C^{*,\tau}$ category. 
\end{lemma}

\begin{kor}
A $C^{*,\tau}$-algebra is semiprojective if and only if its unitization is. 
\end{kor}

\begin{eks}
Since $C((0,1), \id)^\sim \cong C(S^1, \id)$ and the latter is semiprojective, $C((0,1),\id)$ is semiprojective. 
\end{eks}

\subsubsection{Direct sums}

In this section we aim to show the following.

\begin{prop} \label{directSumsp}
If $(A,\tau),(B,\tau)$ are separable and semiprojective $C^{*,\tau}$-algebras, then $(A \oplus B,\tau \oplus \tau)$ is a semiprojective $C^{*,\tau}$-algebra. 
\end{prop}

Before we can do that however, we need to set up some theory. 

\begin{lemma} \label{LiftTwoPosOrthSelfTau}
The relations $0 \leq h,k \leq 1, h = h^\tau, k = k^\tau, hk = 0$ are liftable. 
\end{lemma}
\begin{proof}
Suppose we are given a $\tau$-invariant ideal $J$ in a $C^{*,\tau}$-algebra $B$, and suppose $h,k \in B / J$ satisfy the relations. Let $a = h - k$. Then $a$ is a a self-$\tau$ self-adjoint contraction. Thus we can lift it to a self-adjoint self-$\tau$ contraction in $B$, $\hat{a}$ say. Define $f \colon \R \to \R$ by $f(x) = (x + |x|)/2$. Then we know from $C^*$-algebra theory that $f(\hat{a})$ is a positive contractive lift of $h$, that $f(-\hat{a})$ is a positive contractive lift of $k$, and that $f(\hat{a})f(-\hat{a}) = 0$. Lemma~\ref{tauAndFunctions} tells us that $f(\hat{a})$ and $f(-\hat{a})$ are self-$\tau$.
\end{proof}

\begin{lemma} \label{tauSigmaUnital}
Let $(B,\tau)$ be a $C^{*,\tau}$-algebra. If $h \in (B,\tau)$ is strictly positive in $B$ then so is $h^\tau$. Hence $\Re_\tau(h)$ is strictly positive. 
\end{lemma}
\begin{proof}
Let $\phi \colon B \to \C$ be a linear positive functional. Then we have, writing $\tau$ as a function,  
\[
	\phi(h^\tau) = \phi(\tau(h)) = (\phi \circ \tau)(h).
\]
Since $\tau$ is linear and maps positive elements to positive elements $\phi \circ \tau$ is a positive linear functional. But then if $\phi$ is non-zero we have 
\[
	\phi(h^\tau) = (\phi \circ \tau)(h) > 0. \qedhere
\]
\end{proof}

\begin{kor} \label{realStrictlyPos}
If $(B,\tau)$ is a separable $C^{*,\tau}$-algebra then it contains a self-$\tau$ positive element $h$ such that $\overline{hBh} = B$. 
\end{kor}

A discussion of hereditary subalgebras in the context of real $C^*$-algebras
can be found in \cite{StaceyRealInfinite}.

We are now ready to prove Proposition~\ref{directSumsp}. 

\begin{proof}[Proof of Proposition~\ref{directSumsp}]
Since both $(A,\tau)$ and $(B,\tau)$ are separable we can use Corollary~\ref{realStrictlyPos} to find $h \in A$ and $k \in B$, positive contractions such that $h^\tau = h$, $k^\tau = k$, $\overline{hAh} = A$ and $\overline{kBk} = B$. Suppose we are given a $C^{*,\tau}$-algebra $(D,\tau)$ containing an increasing sequence of $\tau$-invariant ideals $J_1 \subseteq J_2 \subseteq \cdots$ and a $C^{*,\tau}$-homomorphism 
\[
	\phi \colon (A \oplus B,\tau \oplus \tau) \to (D/J,\tau),
\]
where $J = \overline{ \cup_n J_n}$. Let $\hat{h} = \phi((h,0))$ and $\hat{k} = \phi((k,0))$. Since $\hat{h}$ and $\hat{k}$ are orthogonal positive self-$\tau$ contractions we can, by Lemma~\ref{LiftTwoPosOrthSelfTau}, find positive orthogonal self-$\tau$ contractive lifts $\tilde{h}$, $\tilde{k}$ of them in $D$. For each $n \in \N \cup \{ \infty \}$ let $h_n = \pi_n(\tilde{h})$, $k_n = \pi_n(\tilde{k})$, $A_n = \overline{h_n (D / J_n) h_n}$, and $B_n = \overline{k_n (D / J_n) k_n}$. For each $n \in \N \cup \{ \infty \}$ the map $\gamma_n = \tau \vert_{A_n} \oplus \tau \vert_{B_n}$ is a reflection since $h_n$ and $k_n$ are self-$\tau$. Observe that we have
\[
  \overline{\hat{h} (D/J) \hat{h}} = \overline{h_\infty (D/J) h_\infty} = A_\infty \quad \text{ and } \quad \overline{\hat{k} (D/J) \hat{k}} = \overline{k_\infty (D/J) k_\infty)} = B_\infty.
\]
Define for each $n \in \N \cup \{ \infty \}$ a map
\[
	\alpha_n \colon (A_n \oplus B_n, \gamma_n) \to (D / J_n,\tau),
\]
by $\alpha((x,y)) = x + y$. It will be an $C^{*,\tau}$-homomorphism since $h_n k_n = 0$. Noticing that
\[
	\pi(\overline{\tilde{h} D \tilde{h}}) = A_\infty \quad \text{ and } \quad \pi(\overline{\tilde{k} D \tilde{k}}) = B_\infty, 
\]
we see there must be a $C^{*,\tau}$-homomorphism 
\[
	\psi \colon (A \oplus B,\tau \oplus \sigma) \to (A_\infty \oplus B_\infty, \gamma_\infty),
\]
such that $\phi \colon \alpha_\infty \circ \psi$. Hence we get the following commutative diagram for all $n \in \N$
\[
	\xymatrix{
		& (A_n \oplus B_n, \gamma_n) \ar@{->>}[d] \ar[r]^-{\alpha_n} & (D/J_n, \tau) \ar@{->>}[d] \\
		(A \oplus B,\tau \oplus \sigma) \ar@/_2pc/[rr]_-{\phi} \ar[r]^-{\psi} & (A_\infty \oplus B_\infty, \gamma_\infty) \ar[r]^-{\alpha_\infty} & (D/J,\tau)
	}
\]
Since $\gamma$ is a direct sum of two reflections, we can use the semiprojective of $(A,\tau)$ and $(B,\sigma)$, one at a time, to show that $(A \oplus B, \tau \oplus \sigma)$ is semiprojective. 
\end{proof}

\begin{rem}
We observe that we only used $(A,\tau)$ and $(B,\tau)$ separable to get strictly positive real elements $h,k$. So we might as well have assumed that $A$ and $B$ were $\sigma$-unital. Lemma~\ref{tauSigmaUnital} tells us that whether we define $(A,\tau)$ to be $\sigma$-unital when $A$ is or when $(A,\tau)$ contains a strictly positive real element, we get the same class of algebras. 
\end{rem}

The knowledge we have accumulated so far lets us take a small step towards showing that if $X$ is a finite one-dimensional CW-complex then $C(X,\id)$ is semiprojective. 

\begin{prop} \label{spFlower}
If $X$ is a wedge of circles (a bouquet) then the $C^{*,\tau}$-algebra $C(X, \id)$ is semiprojective. 
\end{prop}
\begin{proof}
By assumption 
\[
	C(X, \id) \cong \left( \bigoplus_{i=1}^n C_0((0,1),\id) \right)^{\widetilde{\quad}} ,
\]
for some $n \in \N$. By Proposition~\ref{directSumsp} $\bigoplus_{i=1}^n C_0((0,1),\id)$ is semiprojective since each summand is. So by Corollary~\ref{spIffUnitizationSp} $C(X,\id)$ is semiprojective. 
\end{proof}

The above proposition will later be the basis step of an induction proof. 

\begin{rem}
If $X$ is a wedge of two circles, then we can put a reflection on $C(X)$ by mapping one circle to the other. This reflection is not a direct sum of two reflections on the circle. Hence showing that the $C^{*,\tau}$-algebra it defines is semiprojective requires different techniques than the ones we have just used. 
\end{rem}

%%%%%%%%%%%
%%%%%%%%%%%

%%%%%%%%%%%
%%%%%%%%%%%

\section{Multiplier algebras}

In this section we will study multiplier and corona algebras of $C^{*,\tau}$-algebras. The idea is that we already have multiplier algebras at our disposal. So the main body of work lies in showing that we can extend a reflection on $A$ to a reflection on $M(A)$.

\subsection{A reflection on $M(A)$}

The following theorem is in \cite{kasparov1980}. We present it here with a few more details. 

\begin{thm}
Suppose $(A,\tau)$ is a $C^{*,\tau}$-algebra. There is an operation $\tau$ on $M(A)$ defined by 
\[
	m^{\tau}a=\left(a^{\tau}m\right)^{\tau}, \quad \text{ and } \quad am^{\tau}=\left(ma^{\tau}\right)^{\tau}
\]
for $a$ in $A$ and $m$ in $M(A),$ and $(M(A),\tau)$ is a $C^{*,\tau}$-algebra, and the $C^{*}$-inclusion
\[
	\iota \colon A \rightarrow M(A),
\]
is also a $C^{*,\tau}$-homomorphism. 
\end{thm}
\begin{proof}
Consider for a moment a fixed $m$ in $M(A).$ Define $L \colon A \rightarrow A$ and $R \colon A \rightarrow A$ by
\[
	L(a) = \left(a^{\tau}m\right)^{\tau} \quad \text{ and } \quad R(a)=\left(ma^{\tau}\right)^{\tau}.
\]
For all $a$ and $b$ in $A,$
\[
L(ab) = \left(b^{\tau}a^{\tau}m\right)^{\tau} = \left(a^{\tau}m\right)^{\tau}b = L(a)b,
\]
\[
R(ab) = \left(mb^{\tau}a^{\tau}\right)^{\tau} = a\left(mb^{\tau}\right)^{\tau} = aR(b),
\]
and 
\[
R(a)b = \left(ma^{\tau}\right)^{\tau}b = \left(b^{\tau}ma^{\tau}\right)^{\tau} = a\left(b^{\tau}m\right)^{\tau} = aL(b),
\]
so $(L,M)$ is an element of $M(A),$ which we denote $m^{\tau}$. Notice $m^{\tau}$ is specified within all multipliers by either one of the formulas
\[
m^{\tau}a = \left(a^{\tau}m\right)^{\tau}, \quad \text{ or } \quad am^{\tau} = \left(ma^{\tau}\right)^{\tau}.
\]
We claim that the operation defined above, on all multipliers $m\mapsto m^{\tau}$, makes $M(A)$ a $C^{*,\tau}$-algebra. For any multiplier $m$, and any $a$ in $A$,
\[
m^{\tau\tau}a = \left(a^{\tau}m^{\tau}\right)^{\tau} = \left(ma^{\tau\tau}\right)^{\tau\tau} = ma,
\]
so $\tau\circ\tau=\mathrm{id}$. For $n$ in $M(A)$ and $\alpha$ in $\mathbb{C}$, 
\[
(\alpha m+n)^{\tau}a = \left(a^{\tau}(\alpha m+n)\right)^{\tau} = \alpha m^{\tau}a+n^{\tau}a = (\alpha m^{\tau}+n^{\tau})a,
\]
\[
(mn)^{\tau}a = \left(a^{\tau}mn\right)^{\tau} = ((m^{\tau}a)^{\tau}n)^{\tau} = n^{\tau}(m^{\tau}a) = (n^{\tau}m^{\tau})a,
\]
and
\[
\left(m^{*}\right)^{\tau}a=\left(a^{\tau}m^{*}\right)^{\tau}=\left(ma^{*\tau}\right)^{*\tau}=(a^{*}m^{\tau})^{*}=(m^{\tau})^{*}a,
\]
which means $\tau$ commutes with $*$, is anti-multiplicative, and
$\mathbb{C}$-linear. 

If $a$ is in $A,$ then for any other $b$ in $A$,
\[
\iota(a)^{\tau}b=(b^{\tau}\iota(a))^{\tau}=(b^{\tau}a)^{\tau}=a^{\tau}b=\iota(a^{\tau})b,
\]
so $\iota(a)^{\tau}=\iota(a^{\tau}).$ 
\end{proof}

We now observe that a few more standard constructions automatically respect reflections.

\begin{lemma}
Suppose $A$ is a $C^{*,\tau}$-subalgebra of $B,$ where $(B,\tau)$ is a given $C^{*,\tau}$-algebra. The idealizer
\[
I(A \colon B)=\left\{ b \in B \mid bA+Ab \subseteq A \right\}, 
\]
is self-$\tau,$ and so a $C^{*,\tau}$-subalgebra of $B$ containing $A$ as a self-$\tau$ ideal.
\end{lemma}
\begin{proof}
Suppose $b$ is in the idealizer and $a$ is in $A$. Then $a^{\tau}\in A$ and so $(b^\tau a)^\tau = a^\tau b$ is in $A$. Hence $b^\tau a \in A$. Likewise $(a b^\tau)^\tau \in A$ so $a b^\tau \in A$.
\end{proof}

\begin{thm} \label{ThetaIsTau}
Suppose $(B,\tau)$ is a $C^{*,\tau}$-algebra and $A \triangleleft (B,\tau)$ is a self-$\tau$ ideal. The unique $*$-homomorphism
$\theta \colon B \rightarrow M(A)$ for which $\theta(a)=\iota(a)$ for all $a$ in $A$, is automatically a $*$-$\tau$-homomorphism.
\end{thm}
\begin{proof}
We know $\theta(b)a=ba$ defines the only possible $*$-homomorphism from $B$ to $M(A)$ satisfying $\theta(a)=\iota(a).$ For $b$  in $B$  and $a$ in $A$ we compute
\[
\theta(b)^{\tau}a=\left(a^{\tau}\theta(b)\right)^{\tau}=\left(\theta(a^{\tau}b)\right)^{\tau}=\theta((a^{\tau}b)^{\tau})=\theta(b^{\tau}a)=\theta(b^{\tau})a,
\]
which proves $\theta$ is $\tau$-preserving.
\end{proof}

\begin{lemma} \label{MultMapIsTao}
Let $\varphi \colon (A,\tau)\rightarrow(B,\tau)$ be a proper $*$-$\tau$-homomorphism between $\sigma$-unital $C^{*,\tau}$-algebras.
The unique $*$-homomorphism $\widehat{\varphi} \colon M(A) \to M(B)$ that extends $\varphi$ is actually a $*$-$\tau$-homomorphism.
\end{lemma}

\begin{proof}
The fact that $\varphi$ is proper
tells us $B=\varphi(A)B=B\varphi(A).$ The defining formulas
for $\widehat{\varphi}$ are
\[
\widehat{\varphi}(m)\varphi(a)b=\varphi(ma)b
\]
and 
\[
b\varphi(a)\widehat{\varphi}(m)=b\varphi(am).
\]
Therefore
\begin{align*}
\widehat{\varphi}(m)^{\tau}\varphi(a)b & = ((\varphi(a)b)^{\tau}\widehat{\varphi}(m))^{\tau}\\
 & =(b^{\tau}\varphi(a^{\tau})\widehat{\varphi}(m))^{\tau}\\
 & =(b^{\tau}\varphi(a^{\tau}m))^{\tau}\\
 & =\varphi(m^{\tau}a)b\\
 & =\widehat{\varphi}(m^{\tau})\varphi(a)b. \qedhere
\end{align*}
\end{proof}

We get ``multiplier realization'' for free.

\begin{thm} \label{multRealization}
Let C(E) denote the corona of a $\sigma$-unital $C^{*,\tau}$-algebra $(E,\tau),$ and let $D$ and $N$ be separable $C^{*,\tau}$-subalgebras of $C(E).$ Suppose 
\[
A\subseteq C(E)\cap D^{\prime}\cap N^{\perp},
\]
is a $\sigma$-unital $C^{*,\tau}$-subalgebra. Then the $*$-$\tau$-homomorphism
\[
\theta \colon I(A \colon C(E)\cap D^{\prime}\cap N^{\perp})\rightarrow M(A)
\]
is onto.
\end{thm}
\begin{proof}
We know that $\theta$ is onto, by Corollary 3.2 of \cite{ELPBusby}. All we are asserting here is that this map is now a morphism in the category of $C^{*,\tau}$-algebras.
\end{proof}

\subsection{Corona extendible morphisms}

\begin{Def}
We say a morphism of $C^{*,\tau}$-algebras
$\gamma \colon (A,\tau )\rightarrow (B,\tau)$ is
\emph{corona extendible} if, for every $*$-$\tau$-homomorphism
$\varphi \colon A \rightarrow C(E)$ with $E$
a $\sigma$-unital $C^{*,\tau}$-algebra, there exists a $*$-$\tau$-homomorphism $\widehat{\varphi} \colon B \rightarrow C(E)$ so that $\widehat{\varphi}\circ\gamma=\varphi$.
\end{Def}

\begin{thm} \label{Proj2CE}
Suppose $0 \rightarrow A \rightarrow X \rightarrow P \rightarrow 0$ is a
short-exact sequence of $\sigma$-unital $C^{*,\tau}$-algebras
If $P$ is projective then the inclusion $A\rightarrow X$ is corona extendible.
Moreover, the unitization of this map $\widetilde{A} \rightarrow \widetilde{X}$
is also corona extendible.
\end{thm}

\begin{proof}
Except for the $*$-$\tau$-homomorphism claim, this is
Theorem~3.4 of \cite{ELPBusby} combined with the usual
universal property of a split extension, as in Theorem~7.3.6
of  \cite{loringBook}.  We summarize those proofs and
verify that various maps can be selected to be $*$-$\tau$-homomorphisms.

Since $P$ is projective, the exact sequence has a splitting by a $C^{*,\tau}$-homomorphism $\lambda \colon P\rightarrow X$.
We assume we are given a $*$-$\tau$-homomorphism $\varphi \colon A \rightarrow C(E)$
with $E$ being $\sigma$-unital.  As in the proof of
Theorem~3.4 of \cite{ELPBusby}, we have the commutative
diagram, ignoring for now $\psi_0,$
\[
\xymatrix{
0 \ar[r]	& A\ar@{=}[d]	\ar[r]& X \ar[d]	\ar[r]& P\ar[r] \ar@/_/[l] _{\lambda} \ar@/^/[lddd] ^{\psi_0}& 0 \\
	& A \ar[r]\ar[d]	& M(A)\ar[d] \\
	& \varphi(A)\ar[r] & M(\varphi(A)) \\
	& \varphi(A)\ar@{^{(}->}[r]\ar@{=}[u] & I(\varphi(A) \colon C(E)) \ar@{->>}[u] 
}
\]
where the map $A\rightarrow \varphi(A)$ is the co-restriction of $\varphi$
making it onto. The essential fact that the arrow up from the
idealizer to the multiplier algebra is both surjective and
a $*$-$\tau$-homomorphism is Theorem~\ref{multRealization}.
The map from $X$ to $M(A)$ in the top square is a $*$-$\tau$-homomorphism
by Theorem~\ref{ThetaIsTau}.  The map from $M(A)$ to $M(\varphi(A))$
in the middle square is a $*$-$\tau$-homomorphism by
Lemma~\ref{MultMapIsTao}.  We use the projectivity, in the
$*$-$\tau$-sense, of $P$ to get a $*$-$\tau$-homomorphism $\psi_0$
making the diagram commute.

Following $\psi_0$ by the inclusion into the corona algebra
give us a $*$-$\tau$-homomor-phism %spacing hack
$\psi \colon P\rightarrow C(E)$
such that 
\[
	\psi(p)\varphi(a) = \varphi(\lambda(p)a)
\]
for all $p$ in $P$ and $a$ in $A$. This induces a $*$-homomorphism
\[
	\Psi \colon X\rightarrow C(E)
\]
extending $\varphi$ by 
\[
	\Psi(a+\lambda(p))=\varphi(a)+\psi(p)
\]
which is evidently a $*$-$\tau$-homomorphism.

To get the last claim, we must use more of the power
of Theorem~\ref{multRealization}.
We are given $\widetilde{A} \rightarrow C(E)$ which we
regard as a  $*$-$\tau$-homomorphism $\varphi \colon A \rightarrow C(E)$
together with a projection $p$ in $C(E)$ such that 
$p\varphi(a) = \varphi(a)$ for all $a$ in $A.$  We can
replace $I(\varphi(A) \colon C(E))$ in the big diagram by
\[
I(\varphi(A) \colon C(E)) \cap (1-p)^\perp.
\]
We still have
the needed surjectivity onto $M(\varphi(A))$ and end
up with $\Psi \colon B \rightarrow C(E)$  with the property
$p\Psi(b) = \Psi(b)$ for all $b$ in $B.$
\end{proof}

\begin{kor} \label{coronaGraph}
Suppose $X$ is a compact metrizable space and $Y\subseteq X$ is
a closed subset of $X$ homeomorphic to the closed interval
$[0;1]$. Let $X_{1}$ be the quotient of $X$ obtained by
collapsing $Y$ to a point. The inclusion $C(X_{1},\mathrm{id})\hookrightarrow C(X,\mathrm{id})$ of abelian $C^{*,\tau}$-algebras is corona extendible.
\end{kor}
\begin{proof}
Let $y_0$ denote the point in $Y$
associated to $0$ in $[0;1].$  Let $y_*$ be the point in
$X_1$ that is the image of $Y$ in the quotient map.
We have an exact sequence
\[
\xymatrix{
0 \ar[r] 
		& C_0(X_1 \setminus \{y_*\}) \ar[r] 
			& C_0(X \setminus \{y_0\})  \ar[r] 
				& C_0(0;1] \ar[r]
					& 0 
}
\]
where all $C^*$-algebras are equipped with the trivial $\tau$ operation.
Thus we are done by Theorem \ref{Proj2CE} and Example \ref{projectiveEx}. 
\end{proof}

\subsection{Corona (semi-) projective}

Just as in the $C^*$-case, the work of showing semiprojectivity can be reduced using corona algebras. Most of the proof of the following two theorems can be copied from the proof of \cite[Theorem 14.1.7]{loringBook} if one only remembers to change category. The only change is that we have not studied the Calkin algebra in a $C^{*,\tau}$ setting. To avoid using that, use the corona algebra of $\bigoplus_{n=1}^\infty \widetilde{(A,\tau)}$.

\begin{thm} \label{thm:coronaProj}
Suppose $A$ is a separable $C^{*,\tau}$-algebra. The following are equivalent:
\begin{enumerate}
	\item $A$ is projective;
	\item we can solve the lifting problem for $A$ whenever $\rho$ is the
quotient map $M(E)\rightarrow C(E)$ for a separable $C^{*,\tau}$-algebra
$E$ and $\varphi$ is injective;
	\item we can solve the lifting problem for $A$ whenever $\rho$ is the
quotient map $M(E)\rightarrow C(E)$ for a separable $C^{*,\tau}$-algebra
$E;$
	\item we can solve the lifting problem for $A$ whenever $\rho$ is the
quotient map $B\rightarrow B/I$ for a separable $C^{*,\tau}$-algebra
$B$ and closed $\tau$-closed ideal $I.$
\end{enumerate}
\end{thm}

\begin{thm}
Suppose $A$ is a separable $C^{*,\tau}$-algebra. The following are equivalent:
\begin{enumerate}
	\item $A$ is semiprojective;
	\item we can solve the partial lifting problem for $A$ whenever
$B=M(E)$ for a separable $C^{*,\tau}$-algebra $E$ and
$\overline{\bigcup E_{k}}=E$ for some chain of $\tau$-invariant ideals 
of $E$ and $\varphi$ is injective;
	\item we can solve the partial lifting problem for $A$ whenever
$B=M(E)$ for a separable $C^{*,\tau}$-algebra $E$ and
$\overline{\bigcup E_{k}}=E$ for some chain of $\tau$-invariant ideals 
of $E;$
	\item we can solve the partial lifting problem for $A$ whenever
$B$ is separable.
\end{enumerate}
\end{thm}

Just as in the $C^*$ case, the corona algebra description of (semi-) projectivity let us prove that (semi-) projectivity passes to ideals in certain cases.  

\begin{thm}
Suppose $0\rightarrow I\rightarrow A\rightarrow B\rightarrow0$ is
an exact sequence of separable $C^{*,\tau}$-algebras. If $A$ and
$B$ are projective then $I$ is projective.
\end{thm}
\begin{proof}
We need only lift morphisms of the form $I \rightarrow C(E).$ These extend to morphisms $A \rightarrow C(E),$ and those morphisms lift.
\end{proof}

\begin{thm}
Suppose $0\rightarrow I\rightarrow A\rightarrow B\rightarrow0$ is
an exact sequence of separable $C^{*,\tau}$-algebras. If $A$ is
semiprojective and $B$ is projective then $I$ is semiprojective.
\end{thm}

%%%%%%%%%%%
%%%%%%%%%%%

%%%%%%%%%%%
%%%%%%%%%%%

\section{Functions on graphs} \label{FunGraphSec}

In this section we show semiprojectivity of continuous functions on finite one-dimensional CW-complexes with the trivial reflection. We will sometimes refer to finite one-dimensional CW-complexes as graphs. The proof follows the ideas put forth in \cite{loring96graph}. In that paper semiprojectivity of ``dimension drop graphs'' is shown. Since we have a specific goal in mind, we have chosen to discard the matrix algebras.

\begin{thm}
If $X$ is a finite one-dimensional CW complex, then $C(X, \id)$ is a semiprojective $C^{*,\tau}$-algebra. 
\end{thm}
\begin{proof}
Since semiprojectivity is closed under direct sums, we can assume that $X$ is connected. We will do the proof by induction on the number of vertices in $X$. 

The case where $X$ has only one vertex is Proposition~\ref{spFlower}.

Suppose now any one-dimensional CW complex with $k$ vertices gives rise to a semiprojective $C^{*,\tau}$-algebra. Let $X$ be a one-dimensional CW complex with $k+1$ vertices. Fix two vertices, $v_1$ and $v_2$ say. Let $\tilde{X}$ be a topological copy of $X$. Denote the copies of $v_1$ and $v_2$ in $\tilde{X}$ by $w_1$ and $w_2$ respectively. Choose a continuous function $h_0 \colon \tilde{X} \to [-1;2]$ such that $h_0^{-1}([-1;0])$ consists of the union of closed subintervals, containing $w_1$, of each of the edges adjacent to $w_1$, and such that $h_0^{-1}([1,2])$ consists of the same for the edges adjacent to $w_2$, and also $h_0^{-1}(\{ -1 \}) = \{ w_1 \}$ and $h_0^{-1}(\{ 2 \}) = \{ w_2 \}$. We will identify $X$ with the quotient of $\tilde{X}$ obtained by collapsing $h_0^{-1}([-1;0])$ to one point and $h_0^{-1}([1;2])$ to another. Let $\gamma_X \colon \tilde{X} \twoheadrightarrow X$ be the quotient map. Collapsing $v_1$ and $v_2$ to one point we obtain a space, $Y$ say. Let $\eta \colon X \to Y$ be the quotient map. Collapsing $w_1$ and $w_2$ in $\tilde{X}$ we get a space $\tilde{Y}$, call the quotient map $\tilde{\eta}$. And we can collapse arcs in $\tilde{Y}$ to obtain $Y$, with quotient map $\gamma_Y$ say. Thus we have a nice commuting square of quotient maps
\[
  \xymatrix{
    \tilde{X} \ar@{->>}[r]^-{\gamma_X} \ar@{->>}[d]_-{\tilde{\eta}}	& X \ar@{->>}[d]^-{\eta} \\
    \tilde{Y} \ar@{->>}[r]_-{\gamma_Y}					& Y
  }
\]
We will view $C(X,\id)$, $C(\tilde{Y}, \id)$ and $C(Y,\id)$ as subalgebras of $C(\tilde{X}, \id)$ using the following identifications:
\begin{align*}
	C(X,\id) & \cong \left\{ f \in C(\tilde{X}, \id) \mid \begin{array}{rl} f(x) = f(w_1) &\text{ if }  h_0(x) \leq 0 \\ f(x) = f(w_2) &\text{ if } h_0(x) \geq 1\\ \end{array} \right\}, \\
	C(\tilde{Y}, \id) & \cong \{ f \in C(\tilde{X}, \id) \mid f(w_1) = f(w_2) \}, \\
	C(Y,\id) & \cong \{ f \in C(\tilde{X}, \id) \mid f(x) = f(w_1) \text{ if } h_0(x) \leq 0 \text{ or } h_0(x) \geq 1  \}.
\end{align*} 
Define $h_1 \colon \tilde{X} \to [0,1]$ by 
\[
	h_1(x) = \left\{	\begin{array}{ll}
											0,			& h_0(x) \leq 0 \\
											h_0(x), & 0 \leq h_0(x) \leq 1 \\
											1, 			& 1 \leq h_0(x)
										\end{array} \right..
\]
Note that $h_1$ and $C(Y, \id)$ generate $C(X, \id)$. 

Suppose now that we are given a $C^{*,\tau}$-algebra $(E,\tau)$ containing an increasing sequence of $\tau$-invariant ideals $E_1 \subseteq E_2 \subseteq \cdots$ such that $\overline{\cup_n E_n} = E$, and an injective $C^{*,\tau}$-homomorphism $\phi \colon C(X,\id) \to (C(E),\tau)$. Putting some of the quotient maps and $\phi$ into one diagram, we have the following.
\[
	\xymatrix{
		C(\tilde{Y},\id) \ar[d]_-{(\tilde{\eta})_*} & & \\
		C(\tilde{X},\id) & \ar[l]_-{{\gamma_X}_*} C(X, \id) \ar[r]^-{\phi} & (C(E),\tau)
	} 
\]
where ${-}_*$ denotes the induced maps. Using Corollary~\ref{coronaGraph} repeatedly we get a $C^{*,\tau}$-homomorphism $\hat{\phi} \colon C(\tilde{X}, \id) \to (C(E),\id)$ such that $\phi = \hat{\phi} \circ {\gamma_X}_*$. Using that $\tilde{Y}$ is a one-dimensional CW complex with one vertex less than $X$ we get that $C(\tilde{Y},\id)$ is semiprojective, so we can find an $n \in \N$ and a $C^{*,\tau}$-homomorphism $\psi \colon C(\tilde{Y},\id) \to (M(E) / E_n,\tau)$ such that $\pi_{n,\infty} \circ \psi = \hat{\phi} \circ (\tilde{\eta})_*$. All in all we have the following commutative diagram
\[
	\xymatrix{
		C(\tilde{Y},\id) \ar[rr]^-{\psi} \ar[d]_-{(\tilde{\eta})_*} & & (M(E) / E_n,\tau) \ar@{->>}[d]^{\pi_{n,\infty}} \\
		C(\tilde{X},\id) \ar@/_1pc/[rr]_-{\hat{\phi}} & \ar[l]_-{(\gamma_X)_*} C(X, \id) \ar[r]^-{\phi} & (C(E),\tau)
	}
\]

We will now find a lift of $\hat{\phi}(h_1)$ in $(M(E) / E_n,\tau)$ that is positive contractive self-$\tau$ and commutes with $(\psi \circ (\gamma_Y)_*)(C(Y,\id))$. 
Since $\hat{\phi}(h_1)$ is positive and contractive, we can find a positive and contractive lift.
Averaging this lift with $\tau$ of it, we get a self-$\tau$ positive contractive lift of $\hat{\phi}(h_1)$.
Let us call it $H$.
Define functions $l,m,k \colon [-1;2] \to [0;1]$ by
\begin{align*}
  l(t)	& =	\left\{ \begin{array}{ll}
			  0,	& -1 \leq t \leq 0, \\
			  t,	& 0 \leq t \leq 1, \\
			  2-t,	& 1 \leq t \leq 2
      	   	        \end{array} \right., \\
  m(t)	& =	\left\{ \begin{array}{ll}
			  -t,	& -1 \leq t \leq 0, \\
			  0,	& 0 \leq t \leq 1, \\
			  t-1,	& 1 \leq t \leq 2
      	   	        \end{array} \right., \\
  k(t)	& =	\left\{ \begin{array}{ll}
			  0,	& -1 \leq t \leq 0, \\
			  t,	& 0 \leq t \leq 1, \\
			  1,	& 1 \leq t \leq 2
      	   	        \end{array} \right..
\end{align*}
Observe that $l + m k = k$, that $k \circ h_0 = h_1$, and that $l \circ h_0$ and $m \circ h_0$ both are in $C(\tilde{Y}, \id)$. Hence we can define 
\[
  \tilde{H} = \psi(l \circ h_0) + \psi((m \circ h_0)^{1/2}) H \psi((m \circ h_0)^{1/2}). 
\]
Since all the functions are real valued we get that $\tilde{H}$ is self-adjoint. Since every thing else is self-$\tau$ so is $\tilde{H}$. It is a lift of $\hat{\phi}(h_1)$ since 
\begin{align*}
  \pi_{n, \infty}(\tilde{H})	& =	\hat{\phi}(l \circ h_0) + \hat{\phi}((m \circ h_0)^{1/2}) \hat{\phi}(h_1) \hat{\phi}((m \circ h_0)^{1/2}) \\
    & =	\hat{\phi}( (l \circ h_0) + (m \circ h_0) h_1  ) = \hat{\phi}( (l \circ h_0) + (m \circ h_0) (k \circ h_0) ) \\ 
    & = \hat{\phi}( (l + mk) \circ h_0 ) = \hat{\phi}(k \circ h_0) = \hat{\phi}(h_1).
\end{align*}
By functional calculus one can replace $\tilde{H}$ with $\hat{H} = k(\tilde{H})$ to obtain a positive contractive lift of $\hat{\phi}(h_1)$. By Lemma~\ref{tauAndFunctions} $\hat{H}$ is self-$\tau$. To show that this lift commutes with $(\psi \circ (\gamma_Y)_*)(C(Y,\id))$ it suffices to show that $\tilde{H}$ does. Let $f \in C(Y,\id)$. Then $f (m \circ h_0) = 0$ so we must have 
\[
  \psi((\gamma_Y)_*(f)) \psi(m \circ h_0) = 0.
\]
Hence $\psi((\gamma_Y)_*(f))$ commutes with $\tilde{H}$. 

Let $D = C(Y \times [0,1])$. We have shown that given a $C^{*,\tau}$-homomorphism $\phi \colon C(X, \id) \to (C(E),\tau)$ we can find an $n_0 \in \N$ and a $C^{*}$-homomorphism $\chi \colon D \to M(E) /E_{n_0}$ such that the following diagram commutes
\[
\xymatrix{
  D \ar[r]^-{\chi} \ar@{->>}[d]_{\beta}	& M(E) / E_{n_0} \ar@{->>}[d]^-{\pi_{n_0, \infty}} \\
  C(X) \ar[r]_-{\phi}		& C(E) 
}
\]
Here $\beta$ denotes the map induced by sending $C(Y, \R)$ (inside $D$) to $C(Y, \R)$ (inside $C(X,\R)$) and $h$ to $h_1$. Since $\hat{H}$ is self-$\tau$ Lemma~\ref{tauAndFunctions} gives that $\chi$ is actually $\tau$-preserving. Hence we can view the above diagram as being a commutative diagram in the $\mathbf{C^{*,\tau}}$ category. 

For each $n \geq n_0$ define $\chi_n = \pi_{n_0, n} \circ \chi$ and let $D_n = D / \ker \chi_n$. Then if $n_0 \leq n \leq n'$ we have a surjection $D_n \twoheadrightarrow D_{n'}$. Since $D = C(Y \times [0;1])$ and each $D_n$, $n \geq n_0$, is a quotient of $D$, there must be spaces $Y_n$, $n \geq n_0$, such that $D_n \cong C(Y_n)$. Thus we have an inductive system
\[
  D \twoheadrightarrow C(Y_{n_0}) \twoheadrightarrow C(Y_{n_0 + 1}) \twoheadrightarrow C(Y_{n_0 + 2}) \twoheadrightarrow \cdots
\]
Call the bonding maps $\delta_{n,n'}$. This is an inductive system in the category of $C^*$-algebras, so we can compute the limit as 
\[
 D / \ker (\pi_{n_0,\infty} \circ \chi) \cong (\pi_{n_0,\infty} \circ \chi)(D) = (\phi \circ \beta)(D) = \phi(C(X)) \cong C(X).
\]
Since $X$ is an ANR we can find an $n_1 \geq n_0$ and a $C^*$-homomorphism $\lambda \colon C(X) \to C(Y_{n_1})$ such that $\delta_{n_1,\infty} \circ \lambda = \id$. Clearly $\lambda$ and $\delta_{n_1,\infty}$ are $C^{*,\tau}$-homomorphisms if we equip all the commutative algebras with the identity reflection.

Consider the the following commutative diagram.
\[
\xymatrix{
  (D,\id) \ar[r]^-{\chi} \ar@{->>}[d]_{\delta_{n_1}}			& (M(E) / E_{n_0},\tau) \ar@{->>}[d]^{\pi_{n_0, n_1}} \\
  (C(Y_{n_1}),\id) \ar@{->>}[d]_{\delta_{n_1,\infty}} \ar@{-->}[r]	& (M(E) / E_{n_1},\tau) \ar@{->>}[d]^{\pi_{n_1, \infty}} \\
  C(X,\id) \ar[r]_{\phi}					& (C(E),\tau)
}
\]
Since all the vertical maps are quotient maps, we can fit a $C^{*,\tau}$-homomor-phism % spacing hack
on the dashed arrow in such a way that the diagram continues to commute. Call this homomorphism $\mu$. We claim that $\mu \circ \lambda$ is a lift of $\phi$. To see that, we compute
\[
  \pi_{n_1,\infty} \circ \mu \circ \lambda = \phi \circ \delta_{n_1,\infty} \circ \lambda = \phi. \qedhere
\]
\end{proof} 

%%%%%%%%%%%
%%%%%%%%%%%

%%%%%%%%%%%
%%%%%%%%%%%

\section{Variations on Lin's theorem}
\label{varOnLin}

From here on out we more less just follow the proof in \cite{friisRordam}, modifying their techniques to keep track of reflections. 

In this section we write $M_n$ for $M_n(\C)$.

\subsection{Approximating normal elements} \label{approxNormalSec}

The following two lemmas gives self-$\tau$ versions of known facts. One that $M_n$ has stable rank one, and the second that we can do polar decomposition of invertible elements.

\begin{lemma} \label{stblMatrix}
Let $\tau$ be a reflection on $M_n$. For any $\e > 0$ and any self-$\tau$ matrix $A \in (M_n, \tau)$ we can find a self-$\tau$ invertible matrix $B$ such that $\| A - B \| < \e$.
\end{lemma}
\begin{proof}
If $A$ is invertible there is nothing to prove. So suppose $A$ is not invertible. Consider the path of self-$\tau$ matrices $B_t = (1-t)A + tI$. Define a function $p \colon [0,1] \to \C$ by
\[
	p(t) = \det(B_t).
\]
By definition of $\det$ and $B_t$ the function $p$ is a polynomial in $t$. Since $p(0) = \det(B_0) = \det(A) = 0$ and $p(1) = \det(B_1) = \det(I) = 1$, $p$ is not constant. Hence it has only finitely many zeros. Thus for any $\e > 0$ we can find a $t_0$ such that $0 < t_0 < \e/(\|A - I\|)$ and $p(t_0) \neq 0$. Then $B_{t_0}$ is self-$\tau$ and invertible, and 
\[
	\| A - B_{t_0} \| = \|At_0 - It_0\| \leq t_0 \|A - I\| < \e. \qedhere
\]
\end{proof}

\begin{lemma} \label{polarDecomp}
Let $a$ be a self-$\tau$ invertible element in a unital $C^{*,\tau}$-algebra $(A,\tau)$. Then $a$ can be written as $a = up$ where $u$ is a self-$\tau$-unitary and $p = (a^*a)^{1/2}$.  
\end{lemma}
\begin{proof}
Since $a$ is invertible so is $a^*a$. Therefore we can define $u = a(a^*a)^{-1/2}$ and $p = (a^*a)^{1/2}$. Then $u$ is a unitary and $up = a$. By using Lemma~\ref{tauAndFunctions} and standard functional calculus tricks, we get 
\[
	u^\tau = ((a^* a)^{-1/2})^\tau a^\tau = ((a^* a)^\tau)^{-1/2} a = (aa^*)^{-1/2} a = a (a^* a)^{-1/2} = u \qedhere
\] 
\end{proof}

Let $(n_j)$ be a sequence of natural numbers and let $\tau_j$ be a reflection on $M_{n_j}$. Define 
\[
  (M,\tau) = \prod_j (M_{n_j}, \tau_j), \quad (A,\tau) = \bigoplus_j (M_{n_j},\tau_j).
\]
Let $\pi \colon (M,\tau) \to (M/A,\tau)$ denote the quotient map. 

The remainder of this section is devoted to showing that normal self-$\tau$ elements in $(M/A,\tau)$ are close to normal self-$\tau$ elements with nice spectrum. Each lemma gradually improves the niceness of the spectrum, until proposition \ref{normalGridSpectrum} gives us what we want. 

\begin{lemma} \label{qPolarDecomp}
For any self-$\tau$ element $a \in (M/A, \tau)$ there exists a self-$\tau$ unitary $u \in (M/A,\tau)$ such that $a = up$, where $p = (a^*a)^{1/2}$. 
\end{lemma} 
\begin{proof}
Let $x = (x_j)$ be any self-$\tau$ lift of $a$. Using Lemma~\ref{stblMatrix} we can for all $j \in \N$ find an invertible self-$\tau$ element $y_j \in (M_{n_j}, \tau_j)$ such that $\| x_j - y_j \| < 1/j$. Then the sequence $(y_j)$ is in $(M,\tau)$ and $\pi(y) = \pi(x) = a$. By Lemma~\ref{polarDecomp} we can, for each $j \in \N$, find a self-$\tau$ unitary $v_j \in (M_{n_j}, \tau_j)$ such that $y_j = v_j q_j$, where $q_j = (y_j^* y_j)^{1/2}$. If we let $v = (v_j)$ and $q = (q_j)$ then $y = vq$ and $v$ is a self-$\tau$ unitary. Now put $u = \pi(v)$ and $p = \pi(q)$. Then $a = \pi(y) = \pi(v) \pi(q) = u p$, $u$ is a self-$\tau$ unitary, and 
\[
  p = \pi(q) = \pi((y^* y)^{1/2}) = (\pi(y)^* \pi(y))^{1/2} = (a^* a)^{1/2}. \qedhere
\]
\end{proof}

\begin{lemma} \label{normalCloseToInv}
If $x \in (M/A,\tau)$ is normal and self-$\tau$ then for every $\e > 0$ there is a normal self-$\tau$ invertible element $y \in (M/A,\tau)$ such that $\| x - y \| < \e$. 
\end{lemma}
\begin{proof}
By Lemma~\ref{qPolarDecomp} we can write $x = up$ where $u$ is a self-$\tau$ unitary and $p = (x^*x)^{1/2}$. Since we assumed $x$ to be normal $u$ and $p$ commute by standard functional calculus. Define $y = u(p + (\e/2) I)$, where $I$ is the unit in $M/A$. Since $y$ is the product of two commuting normal and invertible elements it is normal and invertible. By Lemma~\ref{tauAndFunctions} we have
\[
  p^\tau = \bigl( (x^*x)^\tau \bigr)^{1/2} = \bigl( xx^* \bigr)^{1/2} = \bigl( x^* x \bigr)^{1/2} = p.
\]
From that it follows that $y$ is self-$\tau$. Finally we see that 
\[
  \|x - y \| = \| up - (up + (\e/2)u) \| = \|(\e/2)u\| = \e/2 < \e. \qedhere
\]
\end{proof}

\begin{lemma} \label{normalCloseToAny}
Let $\lambda \in \C$ be given. If $x \in (M/A,\tau)$ is normal and self-$\tau$ then for every $\e > 0$ there is a normal self-$\tau$ element $y \in (M/A,\tau)$ with $\lambda \notin \sigma(y)$, and such that $\| x - y \| < \e$. 
\end{lemma}
\begin{proof}
Let $\tilde{x} = x - \lambda I$. Then $\tilde{x}$ is normal and self-$\tau$ so by Lemma~\ref{normalCloseToInv} we can find a normal, self-$\tau$ and invertible $\tilde{y} \in M/A$ such that $\|\tilde{y} - \tilde{x} \| < \e$. Let $y = \tilde{y} + \lambda I $. Then $y$ is normal and self-$\tau$, and 
\[
  \|y - x \| = \| \tilde{y} + \lambda I - x \| = \| \tilde{y} - (x - \lambda I)\| = \| \tilde{y} - \tilde{x} \| < \e.
\]
We note that since $0$ is not in the spectrum of $\tilde{y}$ we have $\lambda \notin \sigma(y)$.
\end{proof}

\begin{lemma} \label{normalManyHoles}
Let $F$ be an at most countable subset of $\C$. If $x \in (M/A,\tau)$ is normal and self-$\tau$ then for every $\e > 0$ there is a normal self-$\tau$ element $y \in (M/A,\tau)$ with $F \cap \sigma(y) = \emptyset$, and such that $\| x - y \| < \e$. 
\end{lemma}
\begin{proof}
Let $X$ be the set of normal and self-$\tau$ elements in $(M/A,\tau)$. This is a closed subset of $M/A$, so it is a complete metric space. Let $F = \{ \lambda_1, \lambda_2, \ldots \}$. For each $n \in \N$ let $U_n$ be the set of self-$\tau$ normal elements in $(M/A,\tau)$ that do not have $\lambda_n$ in their spectrum. By Lemma $\ref{normalCloseToAny}$ all the $U_n$ are dense in $X$. Since the set of invertible elements in a $C^*$-algebra is open all the $U_n$ are open in the relative topology of $X$. By Baire's theorem the set $\bigcap_n U_n$ is dense in $X$. That is, the set of normal self-$\tau$ elements whose spectrum does not contain $F$ is dense in the set of normal self-$\tau$ elements. 
\end{proof}

For any complex number $z$ we denote by $\Re(z)$ and $\Im(z)$ the real and imaginary parts of $z$. For all $\e > 0$ define 
\begin{align*}
  \Gamma_\e & = \{ z \in \C \mid \Re(z) \in \e\Z \text{ or } \Im(z) \in \e\Z \}, \\
  \Sigma_\e & = \{ z \in \C \mid \Re(z) \in \e(\Z + \frac{1}{2}) \text{ and } \Im(z) \in \e(\Z + \frac{1}{2})\}.
\end{align*}

\begin{prop} \label{normalGridSpectrum}
If $x \in (M/A,\tau)$ is normal and self-$\tau$ then for every $\e > 0$ there is a normal self-$\tau$ element $y \in (M/A,\tau)$ with $\sigma(y) \subseteq \Gamma_\e$, and such that $\| x - y \| < \e$. 
\end{prop}
\begin{proof}
By Lemma \ref{normalManyHoles} we can find a normal and self-$\tau$ element $\tilde{y} \in (M/A, \tau)$ with
\[
  \sigma(\tilde{y}) \cap \Sigma_\e = \emptyset, \quad \text{ and } \quad \| \tilde{y} - x \| < \left( 1 - \frac{\sqrt{2}}{2} \right) \e. 
\]
There is a continuous retraction $f \colon \C \setminus \Sigma_\e \to \Gamma_\e$ with $| f(z) - z | < ( 1 - \frac{\sqrt{2}}{2}) \e$ for all $z$. Let $y = f(\tilde{y})$. Then $y$ is normal, has the right spectrum, and is no more than $\e$ away from $x$. By Lemma~\ref{tauAndFunctions} we have 
\[
  y^\tau = f(\tilde{y})^\tau = f(\tilde{y}^\tau) = f(\tilde{y}) = y. \qedhere
\]
\end{proof}

\subsection{The proof of Theorem 1}

We begin by using our knowledge of semiprojective $C^{*,\tau}$-algebras. 

\begin{prop} \label{liftANormal}
Suppose $(A_n,\tau_n)$ is a sequence of $C^{*,\tau}$-algebras. If $x$ is a normal self-$\tau$ element in 
\[
  (Q,\tau) = \prod_{n=1}^\infty (A_n,\tau_n)\left / \bigoplus_{n=1}^\infty (A_n,\tau_n)\right.
\]
with spectrum contained in some finite one dimensional CW complex, then there is a lift of $x$ to a normal self-$\tau$ element in $\prod_{n=1}^\infty (A_n, \tau_n)$. 
\end{prop}
\begin{proof}
Let $\Gamma$ be a finite graph such that $\sigma(x) \subseteq \Gamma$.  By Lemma~\ref{tauAndFunctions} the map $f \mapsto f(x)$ is a $C^{*,\tau}$-homomorphism from $C(\Gamma,\id)$ to $(Q,\tau)$. Since $C(\Gamma, \id)$ is semiprojective, we can find an $m \in \N$ and a normal self-$\tau$ element 
\[
	y \in \prod_{n=1}^\infty (A_n,\tau_n) \left / \bigoplus_{n=1}^m (A_n, \tau)\right. ,
\]
such that $y$ is a lift of $x$. Identifying 
\[
	\prod_{n=1}^\infty (A_n,\tau_n) \left / \bigoplus_{n=1}^m (A_n, \tau) \right.
\]
with 
\[
	\prod_{n=m}^\infty (A_n, \tau),
\]
we see that if we pad $y$ with leading zeros we get a self-$\tau$ and normal lift of $x$ in $\prod_{n=1}^\infty (A_n, \tau_n)$.  
\end{proof}

We are now ready to prove real versions of Lin's theorem.
First we do the case of normal matrices. 

\begin{thm} \label{normalLin}
For every $\e > 0$ there is a $\delta > 0$ such that for any $n \in \N$, any reflection $\tau$ on $M_n$ and self-$\tau$ matrix $X \in (M_n, \tau)$ with $\|X\| \leq 1$ and
\[
  \|X^* X - X X^*\| < \delta,
\]
there exists a normal self-$\tau$ matrix $X' \in (M_n, \tau)$ with 
\[
  \| X - X' \|< \e. 
\] 
\end{thm}
\begin{proof}
Suppose there was an $\e$ that had no accompanying $\delta$. Then there must exist a sequence $(n_j)$ of natural numbers, reflections $\tau_j$ on $M_{n_j}$, and self-$\tau$ contractive matrices $X_j \in (M_{n_j}, \tau_j)$ such that 
\[
  \| X_j^* X_j - X_j X_j^* \| \to 0,
\]
but every $X_j$ is at least $\e$ away from all normal self-$\tau$ matrices in $(M_{n_j}, \tau_j)$. 

Let, as in Section~\ref{approxNormalSec}, 
\[
  (M,\tau) = \prod_j (M_{n_j}, \tau_j), \quad (A,\tau) = \bigoplus_j (M_{n_j},\tau_j).
\] 
Let $x = (X_j)$ and let $y = \pi(x)$, where $\pi$ is the quotient map from $(M,\tau)$ to $(M/A, \tau)$. Then $y$ is a normal and self-$\tau$ element. By Proposition \ref{normalGridSpectrum} we can find a normal self-$\tau$ element $z \in (M/A, \tau)$ with spectrum contained in a finite graph and $\| y - z \| < {\e/4}$.
Using Proposition~\ref{liftANormal} we can find a normal self-$\tau$ element $x' \in (M,\tau)$ such that $\pi(x') = z$. The definition of the norm in $(M/A, \tau)$ tells us that there exists $(A_j) = a \in A$ such that
\[
  \| (x - x') - a \| = \| y - z \| + \e/4 < \e / 2. 
\]
Now pick a $j_0$ such that $\| A_{j_0} \| < \e / 2$. Then we have 
\[
  \| X_{j_0} - X_{j_0}' \| \leq \| (X_{j_0} - X_{j_0}') - A_{j_0} \| + \|A_{j_0}\| < \|(x - x') - a\| + \e/2 < \e.
\]
Which contradicts our assumption about all the $X_j$ being at least $\e$ away from any normal self-$\tau$ element. 
\end{proof}

\begin{thm} \label{selfTauLin}
For every $\e > 0$ there is a $\delta > 0$ such that for any $n \in \N$, any reflection $\tau$ on $M_n$ and any pair $A,B \in (M_n, \tau)$ of self-adjoint, self-$\tau$ matrices such that $\|A\|, \|B\| \leq 1$ and 
\[
  \|AB - BA\| < \delta,
\]
there exists a commuting pair $A',B' \in (M_n, \tau)$ of self-adjoint and self-$\tau$ matrices with 
\[
  \| A - A' \| + \| B - B' \| < \e. 
\]
\end{thm}
\begin{proof}
Let $\e > 0$ be given. Use Theorem~\ref{normalLin} to find a $\delta$ matching $\e/2$. Let $A, B \in (M_n,\tau)$ be given as in the theorem. Define $X = (A + i B)/2$. Then 
\[
	\| XX^* - X^*X \| = \frac{1}{2} \| AB - BA \| < \delta. 
\]
Hence we can find a normal self-$\tau$ matrix $X' \in M_n$ such that $\|X - X'\| < \e/2$. Now let 
\[
	A' = X' + {X'}^* \quad \text{ and } \quad B' = -i(X' - {X'}^*).
\] 
Then $A'$ and $B'$ are self-adjoint and self-$\tau$. Since $X'$ is normal they commute. As 
\[
	\| A - A'\| = \| (X + X^*) - (X' + {X'}^*) \| \leq  \|X - X'\| + \| X^* - {X'}^*\| < \e,
\]
and likewise for $\|B - B'\|$, $A'$ and $B'$ show that we can approximate $A$ and $B$ by commuting self-adjoint, self-$\tau$ matrices.
\end{proof}

Setting $\tau$ equal the transpose in Theorem~\ref{selfTauLin} we
obtain the extension of Lin's theorem to real matrices.  Using
the dual operation, $\tau=\sharp,$ we obtain the extension of
Lin's theorem to self-dual matrices.

%%%%%%%%%%%
%%%%%%%%%%%

%%%%%%%%%%%
%%%%%%%%%%%

\section{TR rank one} \label{TRrankOne}

The proof of Theorem \ref{mainThm} generalizes to real $C^{*}$-algebras beyond the matrices.
In particular, we get versions of Theorem \ref{mainThm} that apply to paths of almost commuting matrices.

We will need a replacement for Lemma \ref{stblMatrix} regarding the approximation of self-dual matrices by invertible self-dual matrices.
We must introduce a concept along the lines of stable rank and real rank. 

\subsection{TR rank one}

Recall that real rank zero \cite{BrownPedersenRRzero} and stable rank one \cite{Rieffel_tsr} are defined in terms of density of the invertibles within a class of elements. 
For stable rank one the class is all elements,  while for real rank zero the class is restricted by requiring $a^{*}=a$.
When $(A,\tau)$ is a $C^{*,\tau}$-algebra we can look at many more symmetry restrictions, such as $a^{*}=a^{\tau}=a$ or $a^{*}=-a^{\tau}$.
Rank based on the elements with $a^{*}= a^{\tau}$ was considered in \cite{StaceyRealForUHF}.

Looking at Lemma \ref{stblMatrix} we see we need to consider invertibles within the class of elements with $a^{\tau}=a$.
To the resulting definition of rank we give a name that reflects the role of time reversal symmetry in motivating this work.

\begin{Def}
A unital $C^{*,\tau}$-algebra $A$ is said to have \emph{TR rank one} if the invertible elements of $A$ such that $a^{\tau}=a$ are dense in the set of elements of $A$ such that $a^{\tau}=a$. 
A non-unital $C^{*,\tau}$-algebra is said to have TR rank one when its unitization has TR rank one.
An $R^{*}$-algebra has TR rank one when its complexification has TR rank one.
\end{Def}

It is obvious that a direct sum of unital $C^{*,\tau}$-algebras has TR rank one if and only if each factor has TR rank one. Therefore there is no conflict in the first and second sentences of the definition.

We give a few examples of TR rank one algebras.

\begin{lemma}
The $C^{*,\tau}$-algebra $C(S^{1})$
with $f^{\tau} = f$  has TR rank one.
\end{lemma}
\begin{proof}
This is just the usual assertion that a function from the circle to the plane can be perturbed to miss the origin.
\end{proof}

\begin{lemma}
The $C^{*,\tau}$-algebra $C(S^{1}) \otimes M_{2}(\mathbb{C})$ with $f^{\tau}(x)=\left(f(x)\right)^{\sharp}$ has TR rank one.
\end{lemma}
\begin{proof}
The only two-by-two matrices that are self-dual are the scalar matrices, so this follows from the previous lemma.
\end{proof}

It is useful for us  that TR rank one is preserved when adding matrices. 

\begin{prop}[{See \cite[Proposition 3]{RobertsonStableRange}}]
Let $(A,\tau)$ be a $C^{*,\tau}$-algebra with identity. 
If $(A,\tau)$ has TR rank one then $(M_{n}(A),\tau \otimes T)$ has TR rank one.
\end{prop}
\begin{proof}
As an induction hypothesis assume that $M_{n}(A)$ has TR rank one (it is the assumption of the proposition that the basis of the induction is true).
Suppose we have $X=X^{\tau \otimes T}$ in $M_{n+1}(A)$.
In block form
\[
	X = \begin{pmatrix} a & B \\ B^{\tau T} & D \end{pmatrix},
\]
where $a^{\tau}=a$ in $A$ and $D^{\tau}=D$ in $M_{n}(D)$.
Here the meaning of $B^{\tau T}$ is to apply $\tau$ componentwise and then take matrix transpose. 
Let 
\[
	X_0 = \begin{pmatrix} a_{0} & B \\ B^{\tau T} & D \end{pmatrix},
\]
with $a_{0}^{\tau}=a_{0}$ being invertible and close to $a$.
Since $(a_{0}^{-1})^{\tau}=a_{0}^{-1}$ we can find an $E \in M_n(A)$ arbitrarily close to $D-B^{\tau T}a_{0}^{-1}B$ that is invertible and with $E^{\tau}=E$.
Let $D_{0}=E+B^{\tau T}a_{0}^{-1}B$ and notice that
\[
	X_1 = \begin{pmatrix} a_{0} & B \\ B^{\tau T} & D_{0} \end{pmatrix}
\]
will be close to $X$ and invariant under $\tau \otimes T$.
We are done if we show that $X_1$ is invertible. 
To this end note that 
\[
	\begin{pmatrix} 1 & -a_{0}^{-1}B \\ 0 & 1 \end{pmatrix} \quad \text{ and } \quad \begin{pmatrix} a_{0} & 0 \\
0 & E \end{pmatrix}
\]
are invertible. 
Since 
\[
	\begin{pmatrix} 1 & -a_{0}^{-1}B \\ 0 & 1 \end{pmatrix}^{\tau \otimes T} X_1 \begin{pmatrix} 1 & -a_{0}^{-1}B \\ 0 & 1\end{pmatrix} = \begin{pmatrix} a_{0} & 0 \\ 0 & E \end{pmatrix},
\]
we can write $X_1$ as a product of invertibles. 
Hence it is invertible. 
\end{proof}

Although we have no need for it here we expand slightly on how TR rank one behaves with respect to matrices. 

\begin{lemma}[{see \cite[lemma 3.4]{Rieffel_tsr}}]
Let $(A,\tau)$ be a unital $C^{*,\tau}$-algebra and let $n \in \N$. 
If $(M_{n+1}(A), \tau \otimes T)$ has TR rank one, then $(A,\tau)$ has TR rank one. 
\end{lemma}
\begin{proof}
Let a self-$\tau$ element $t \in A$ and an $0 < \e < 1$ be given. 
Define 
\[
	X = \begin{pmatrix}t & 0 \\ 0 & I_n \end{pmatrix} \in M_{n+1}(A),
\]
where $I_n$ is the identity element in $M_n(A)$. 
By assumption we can find an invertible $Y \in M_{n+1}(A)$ within $\e$ of $X$ and such that $Y = Y^{\tau \otimes T}$. 
In block form $Y$ looks like
\[
	Y = \begin{pmatrix}a & B \\ B^{\tau T} & D \end{pmatrix}
\]
Notice that $a^\tau = a$ and $D^{\tau \otimes T} = D$.
Computing as in \cite{Rieffel_tsr} we see that $D$ and $t_0 = a - BD^{-1}B^{\tau T}$ are invertible and that $t_0$ is within $\e(1-\e)^{-1}$ of $t$. 
Hence we are done if we can show that $t_0^{\tau} = t_0$. 
Since $a$ is self-$\tau$ it suffices to show that $BD^{-1}B^{\tau T}$ is self-$\tau$.
For this write $D^{-1} = (d_{ij})$ and $B = (b_1 b_2 \ldots b_n)$.
The symmetry on $D^{-1}$ translates to $d_{ij}^{\tau} = d_{ji}$. 
We compute 
\begin{align*}
	(BD^{-1}B^{\tau T})^\tau & = \left( \sum_{i,j=1}^n b_i d_{ij} b_j^\tau \right)^\tau = \sum_{i,j}^n b_j d_{ij}^\tau b_i^\tau \\
					 & = \sum_{i,j}^n b_j d_{ji} b_i^\tau = BD^{-1}B^{\tau T}. \qedhere
\end{align*}
\end{proof}

Combining the two previous results we get:

\begin{thm}
Let $(A,\tau)$ be a unital $C^{*,\tau}$-algebra. 
The following are equivalent:
\begin{enumerate}
	\item $(A,\tau)$ has TR rank one.
	\item $(M_n(A),\tau \otimes T)$ has TR rank one for all $n \in \N$.
	\item $(M_n(A),\tau \otimes T)$ has TR rank one for some $n \in \N$.
\end{enumerate}
\end{thm}

\begin{kor} \label{TRLoopSpaces}
Both $C(S^{1},M_{n}(\mathbb{C}))$ with $\tau$-operation
\[
f^{\tau}(x)=\left(f(x)\right)^{T}
\]
 and $C(S^{1},M_{2N}(\mathbb{C}))$ with $\tau$-operation
\[
f^{\tau}(x)=\left(f(x)\right)^{\sharp}
\]
have TR rank one.
\end{kor}

\begin{rem}
We could have added the $\tau$-operation that exchanges two summands that are each $C(S^{1},M_{n}(\mathbb{C}))$.
In terms of the underlying $R^{*}$-algebras, we have thus the fact that
\[
C\left(S^{1},M_{n}(\mathbb{R})\right),\ C\left(S^{1},M_{n}(\mathbb{C})\right),\ C\left(S^{1},M_{n}(\mathbb{H})\right)
\]
all have TR rank one.
\end{rem}

\subsection{Generalizing the Friis-R{\o}rdam Theorem}

In the terms of this section, Lemma \ref{stblMatrix} just says that $(M_n(\C), T)$ and $(M_{2n}(\C), \sharp)$ have TR rank one. 
All of section \ref{approxNormalSec} was built on that lemma and the structured polar decomposition. 
Hence one can redo the work starting with the assumption of TR rank 1 instead of Lemma \ref{stblMatrix}.
That work yields the following.

\begin{prop} \label{TRGrid}
Let $(A_n, \tau_n)$ be a sequence of unital TR rank one $C^{*,\tau}$-algebras. 
Define $(M,\tau) = (\prod A_n, \prod \tau_n)$ and $A = \bigoplus A_n$.
If $x \in (M/A,\tau)$ is normal and self-$\tau$ then for every $\e > 0$ there is a normal self-$\tau$ element $y \in (M/A,\tau)$ with $\sigma(y) \subseteq \Gamma_\e$, and such that $\| x - y \| < \e$, where 
\[
	\Gamma_\e = \{ z \in \C \mid \Re(z) \in \e\Z \text{ or } \Im(z) \in \e\Z \}.
\]  
\end{prop}

With that at hand, we get a much more general theorem. 

\begin{thm} \label{biggerThm}
For all $\e > 0$ there exists a $\delta > 0$ such that for all $C^{*,\tau}$-algebras of TR rank one the following holds: Whenever $a,b$ are in $A$ with $a^* = a = a^\tau$, $b^* = b = b^\tau$ and $\|ab - ba\| < \delta$ there exists $a',b'$ in $A$ such that $a'^* = a' = a'^\tau$, $b'^* = b' = b'^\tau$, $a'b' = b'a'$, and
\[
	\|a - a'\|, \|b - b'\| < \e.
\]
\end{thm}
\begin{proof}
Just as it was for matrices, it suffices to prove that a self-$\tau$ almost normal element is close to a self-$\tau$ normal element. 

We reduce to the case where $A$ is unital. 
To do this, let an element $x$ in a non-unital $C^{*,\tau}$-algebra $B$ be given. 
Suppose we can find a normal and self-$\tau$ element $y \in \tilde{B}$ that is close to $x$, say $\| y - x \| < \eta$.
Write $y = \alpha 1 + z$ with $\alpha \in \C$ and $z \in B$. 
Now we notice that $z$ must be normal and self-$\tau$.
Further we have that $\| y - z \| = |\alpha| < \eta$. 
Hence $z$ is a normal and self-$\tau$ element that is at most $2 \eta$ away from $x$.

It remains to prove the theorem for $A$ unital. 
The proof of this is just as the proof for matrices given when we proved Theorem \ref{normalLin}, only we should reference Proposition \ref{TRGrid} instead of Proposition \ref{normalGridSpectrum}.
\end{proof}

As an application, we get nice perturbation properties for paths of matrices. 

\begin{thm} \label{realLinPaths}
For all $\e > 0$ there exists a $\delta > 0$ such that for all $n \in \N$ the following holds:
Whenever $A_t,B_t$ are two closed paths of $n$-by-$n$, contractive, self-adjoint, real (resp. self-dual) matrices such
that $\|A_t B_t - B_t A_t\| < \delta$ for all $t$ there exists closed paths of $n$-by-$n$,  self-adjoint, real (resp. self-dual) matrices $A'_t,B'_t$ such that $A'_tB'_t = B'_tA_t'$ for all $t$ and
\[
	\|A_t - A_t'\|, \|B_t - B_t'\| < \e
\]
for all $t$.
\end{thm}
\begin{proof}
Combine Corollary \ref{TRLoopSpaces} with Theorem \ref{biggerThm}.
\end{proof}

\begin{rem}
Theorem \ref{realLinPaths} is also true for paths that need not be closed.
\end{rem}

%%%%%%%%%%%
%%%%%%%%%%%

%%%%%%%%%%%
%%%%%%%%%%%

%\nocite{*}
\providecommand{\bysame}{\leavevmode\hbox to3em{\hrulefill}\thinspace}
\providecommand{\MR}{\relax\ifhmode\unskip\space\fi MR }
\providecommand{\MRhref}[2]{%
  \href{http://www.ams.org/mathscinet-getitem?mr=#1}{#2}
}
\providecommand{\href}[2]{#2}


\begin{thebibliography}{10}

\bibitem{altland1997nonstandard}
Alexander Altland and Martin~R Zirnbauer, \emph{Nonstandard symmetry classes in
  mesoscopic normal-superconducting hybrid structures}, Physical Review B
  \textbf{55} (1997), no.~2, 1142.

\bibitem{atiyah66reality}
M.~F. Atiyah, \emph{{$K$}-theory and reality}, Quart. J. Math. Oxford Ser. (2)
  \textbf{17} (1966), 367--386. \MR{0206940 (34 \#6756)}

\bibitem{bhatia2007spectral}
Rajendra Bhatia, \emph{Spectral variation, normal matrices, and finsler
  geometry}, Mathematical Intelligencer \textbf{29} (2007), no.~3, 41--46.

\bibitem{blackadar1985shape}
Bruce Blackadar, \emph{Shape theory for {$C^\ast$}-algebras}, Math. Scand.
  \textbf{56} (1985), no.~2, 249--275. \MR{813640 (87b:46074)}

\bibitem{blackadar2004semiprojectivity}
\bysame, \emph{Semiprojectivity in simple {$C^*$}-algebras}, Operator algebras
  and applications, Adv. Stud. Pure Math., vol.~38, Math. Soc. Japan, Tokyo,
  2004, pp.~1--17. \MR{2059799 (2005g:46101)}

\bibitem{BoersemaUnitedKtheory}
Jeffrey~L. Boersema, \emph{Real {$C^*$}-algebras, united {$K$}-theory, and the
  {K}\"unneth formula}, $K$-Theory \textbf{26} (2002), no.~4, 345--402.
  \MR{1935138 (2004b:46072)}

\bibitem{BrownPedersenRRzero}
Lawrence~G. Brown and Gert~K. Pedersen, \emph{{$C^*$}-algebras of real rank
  zero}, J. Funct. Anal. \textbf{99} (1991), no.~1, 131--149. \MR{1120918
  (92m:46086)}

\bibitem{cardoso1998blind}
J.F. Cardoso, \emph{Blind signal separation: statistical principles},
  Proceedings of the IEEE \textbf{86} (1998), no.~10, 2009--2025.

\bibitem{DadarlatShapeAndE}
Marius D{\u{a}}d{\u{a}}rlat, \emph{Shape theory and asymptotic morphisms for
  {$C^*$}-algebras}, Duke Math. J. \textbf{73} (1994), no.~3, 687--711.
  \MR{1262931 (95c:46117)}

\bibitem{circleHom}
Alessandro Derango, \emph{C*-algebras associated with homeomorphisms of the
  unit circle}, ProQuest LLC, Ann Arbor, MI, 2000, Thesis (Ph.D.)--University
  of Toronto (Canada). \MR{2701546}

\bibitem{dyson1962threefold}
Freeman~J Dyson, \emph{The threefold way. algebraic structure of symmetry
  groups and ensembles in quantum mechanics}, Journal of Mathematical Physics
  \textbf{3} (1962), no.~1199.

\bibitem{EffrosKaminkerShape}
E.~G. Effros and J.~Kaminker, \emph{Homotopy continuity and shape theory for
  {$C^\ast$}-algebras}, Geometric methods in operator algebras ({K}yoto, 1983),
  Pitman Res. Notes Math. Ser., vol. 123, Longman Sci. Tech., Harlow, 1986,
  pp.~152--180. \MR{866493 (88a:46082)}

\bibitem{ELPBusby}
S{\o}ren Eilers, Terry~A. Loring, and Gert~K. Pedersen, \emph{Morphisms of
  extensions of {$C^*$}-algebras: pushing forward the {B}usby invariant}, Adv.
  Math. \textbf{147} (1999), no.~1, 74--109. \MR{1725815 (2000j:46104)}

\bibitem{glashoff2012multimodal}
Davide Eynard, Klaus Glashoff, Michael~M Bronstein, and Alexander~M Bronstein,
  \emph{Multimodal diffusion geometry by joint diagonalization of laplacians},
  arXiv preprint arXiv:1209.2295 (2012).

\bibitem{freed2012twisted}
Daniel~S Freed and Gregory~W Moore, \emph{Twisted equivariant matter}, Annales
  Henri Poincar{\'e}, Springer, pp.~1--97.

\bibitem{friisRordam}
Peter Friis and Mikael R{\o}rdam, \emph{Almost commuting self-adjoint
  matrices---a short proof of {H}uaxin {L}in's theorem}, J. Reine Angew. Math.
  \textbf{479} (1996), 121--131. \MR{1414391 (97i:46097)}

\bibitem{glashoff2013almost}
Klaus Glashoff and Michael~M Bronstein, \emph{Matrix commutators: their
  asymptotic metric properties and relation to approximate joint
  diagonalization}, www.cs.technion.ac.il/~mbron/publications.html, 2013.

\bibitem{goodearlBook}
K.~R. Goodearl, \emph{Notes on real and complex {$C^{\ast} $}-algebras}, Shiva
  Mathematics Series, vol.~5, Shiva Publishing Ltd., Nantwich, 1982. \MR{677280
  (85d:46079)}

\bibitem{gygi2003computation}
F.~Gygi, J.L. Fattebert, and E.~Schwegler, \emph{Computation of maximally
  localized wannier functions using a simultaneous diagonalization algorithm},
  Computer physics communications \textbf{155} (2003), no.~1, 1--6.

\bibitem{hastings2008topology}
M.~B. Hastings, \emph{Topology and phases in fermionic systems}, J. Stat. Mech.
  Theory Exp. (2008), no.~1, L01001, 8. \MR{2373523 (2009d:81434)}

\bibitem{hastings2009making}
\bysame, \emph{Making almost commuting matrices commute}, Comm. Math. Phys.
  \textbf{291} (2009), no.~2, 321--345. \MR{2530163 (2010m:15032)}

\bibitem{HastingsLoringWannier}
M.B. Hastings and T.A. Loring, \emph{Almost commuting matrices, localized
  wannier functions, and the quantum hall effect}, Journal of Mathematical
  Physics \textbf{51} (2010), 015214.

\bibitem{HastLorTheoryPractice}
\bysame, \emph{Topological insulators and {$C^*$}-algebras: {T}heory and
  numerical practice}, Arxiv preprint arXiv:1012.1019 (2010).

\bibitem{kasparov1980}
G.~G. Kasparov, \emph{Hilbert {$C^{\ast} $}-modules: theorems of {S}tinespring
  and {V}oiculescu}, J. Operator Theory \textbf{4} (1980), no.~1, 133--150.
  \MR{587371 (82b:46074)}

\bibitem{kovnatsky2013coupled}
Artiom Kovnatsky, Michael~M Bronstein, Alexander~M Bronstein, Klaus Glashoff,
  and Ron Kimmel, \emph{Coupled quasi-harmonic bases}, Computer Graphics Forum,
  vol.~32, Wiley Online Library, 2013, pp.~439--448.

\bibitem{liBook}
Bingren Li, \emph{Real operator algebras}, World Scientific Publishing Co.
  Inc., River Edge, NJ, 2003. \MR{1995682 (2004k:46100)}

\bibitem{LinAlmostCommutingHermitian}
Huaxin Lin, \emph{Almost commuting selfadjoint matrices and applications},
  Operator algebras and their applications ({W}aterloo, {ON}, 1994/1995),
  Fields Inst. Commun., vol.~13, Amer. Math. Soc., Providence, RI, 1997,
  pp.~193--233. \MR{1424963 (98c:46121)}

\bibitem{lindner2011floquet}
Netanel~H Lindner, Gil Refael, and Victor Galitski, \emph{Floquet topological
  insulator in semiconductor quantum wells}, Nature Physics \textbf{7} (2011),
  no.~6, 490--495.

\bibitem{LoringHastingsEPL}
T.A. Loring and M.B. Hastings, \emph{Disordered topological insulators via
  {$C^*$}-algebras}, EPL (Europhysics Letters) \textbf{92} (2010), 67004.

\bibitem{LoringLogOfUnitary}
Terry~A. Loring, \emph{Computing a logarithm of a unitary matrix with general
  spectrum}, arxiv:1203.6151.

\bibitem{LoringQuantKth}
\bysame, \emph{Quantitative $k$-theory and spin chern numbers},
  arxiv:1302.0349.

\bibitem{loring96graph}
\bysame, \emph{Stable relations. {II}. {C}orona semiprojectivity and
  dimension-drop {$C^*$}-algebras}, Pacific J. Math. \textbf{172} (1996),
  no.~2, 461--475. \MR{1386627 (97c:46070)}

\bibitem{loringBook}
\bysame, \emph{Lifting solutions to perturbing problems in {$C^*$}-algebras},
  Fields Institute Monographs, vol.~8, American Mathematical Society,
  Providence, RI, 1997. \MR{1420863 (98a:46090)}

\bibitem{LoringQuaternions}
\bysame, \emph{Factorization of matrices of quaternions}, Exposition. Math.
  \textbf{30} (2012), no.~3, 250--267.

\bibitem{LorSorensenTorus}
Terry~A. Loring and Adam P.~W. S{\o}rensen, \emph{Almost commuting unitary
  matrices related to time reversal}, Comm. Math. Phys. (to appear),
  arxiv:1107.4187.

\bibitem{may99stable}
J.~P. May, \emph{Stable algebraic topology, 1945--1966}, History of topology,
  North-Holland, Amsterdam, 1999, pp.~665--723. \MR{1721119 (2000m:55002)}

\bibitem{ogata2013approximating}
Yoshiko Ogata, \emph{Approximating macroscopic observables in quantum spin
  systems with commuting matrices}, Journal of Functional Analysis (2013).

\bibitem{PalmerRealCstar}
T.~W. Palmer, \emph{Real {$C\sp*$}-algebras}, Pacific J. Math. \textbf{35}
  (1970), 195--204. \MR{0270162 (42 \#5055)}

\bibitem{qi2010quantum}
Xiao-Liang Qi and Shou-Cheng Zhang, \emph{The quantum spin hall effect and
  topological insulators}, Physics Today \textbf{63} (2010), no.~1, 33--38.

\bibitem{Rieffel_tsr}
Marc~A. Rieffel, \emph{Dimension and stable rank in the {$K$}-theory of
  {$C^{\ast}$}-algebras}, Proc. London Math. Soc. (3) \textbf{46} (1983),
  no.~2, 301--333. \MR{693043 (84g:46085)}

\bibitem{RobertsonStableRange}
A.~Guyan Robertson, \emph{Stable range in {$C^{\ast} $}-algebras}, Math. Proc.
  Cambridge Philos. Soc. \textbf{87} (1980), no.~3, 413--418. \MR{556921
  (82h:46079)}

\bibitem{StaceyRealForUHF}
P.~J. Stacey, \emph{Stability of involutory {$*$}-antiautomorphisms in {UHF}
  algebras}, J. Operator Theory \textbf{24} (1990), no.~1, 57--74. \MR{1086544
  (92g:46077)}

\bibitem{StaceyRealInfinite}
\bysame, \emph{Real structure in purely infinite {$C^*$}-algebras}, J. Operator
  Theory \textbf{49} (2003), no.~1, 77--84. \MR{1978322 (2004c:46111)}

\bibitem{vonNeumannQuantumErgodic}
J.~von Neumann, \emph{Beweis des ergodensatzes und des {$H$}-theorems in der
  neuen mechanik}, Zeitschrift f\:ur Physik A Hadrons and Nuclei \textbf{57}
  (1929), 30--70, 10.1007/BF01339852.

\bibitem{Weigand2012oneDim}
Daniel Weigand, \emph{Entanglement entropy in 1d noninteracting fermionic
  systems)}, Master's thesis, 2012,
  http://www.physik.rwth-aachen.de/institute/institut-fuer-quanteninformation/student-projects/past-theses/.

\bibitem{zabrodin2006matrix}
A.~Zabrodin, \emph{Matrix models and growth processes: from viscous flows to
  the quantum {H}all effect}, Applications of random matrices in physics, NATO
  Sci. Ser. II Math. Phys. Chem., vol. 221, Springer, Dordrecht, 2006,
  pp.~261--318. \MR{2232116 (2007d:82081)}

\end{thebibliography}
\end{document}